\documentclass[psamsfonts]{amsart}

\usepackage{amssymb,amsfonts,amscd}
\usepackage[backref,pagebackref,pdftex,hyperindex]{hyperref}
\usepackage[all,cmtip, arc]{xy}
\usepackage{enumerate}
\usepackage{amsmath,amsthm,fancyhdr,graphicx,bbm,cancel,color,mathrsfs, yfonts}
\usepackage{xypic}
\usepackage{mathrsfs}
\usepackage{xcolor}
\usepackage{tikz-cd}
\usepackage{rotating}
\usepackage{comment}
\usepackage{stmaryrd}
\newtheorem{thm}{Theorem}[section]
\newtheorem{cor}[thm]{Corollary}
\newtheorem{prop}[thm]{Proposition}
\newtheorem{lem}[thm]{Lemma}
\newtheorem{conj}[thm]{Conjecture}
\newtheorem{quest}[thm]{Question}

\theoremstyle{definition}
\newtheorem{defn}[thm]{Definition}

\newtheorem{exmp}[thm]{Example}

\newtheorem*{exer*}{Exercise}

\theoremstyle{remark}
\newtheorem{rmk}[thm]{Remark}

\newcommand{\p}{\partial}
\newcommand{\pbar}{\overline{\partial}}

\newcommand{\C}{\mathbb C}
\newcommand{\Q}{\mathbb Q}
\newcommand{\R}{\mathbb R}

\newcommand{\CP}{\mathbb{CP}}

\newcommand{\e}{\varepsilon}

\newcommand{\vp}{\varphi}

\newcommand{\vol}{\mathrm{vol}}

\providecommand{\abs}[1]{\left\lvert#1\right\rvert}
\providecommand{\norm}[1]{\left\lVert#1\right\rVert}
\newcommand{\ov}[1]{\overline{#1}}

\newcommand{\PSH}{\mathrm{PSH}}

\newcommand{\ti}[1]{\widetilde{#1}}

\makeatletter
\let\c@equation\c@thm
\makeatother
\numberwithin{equation}{section}

\title{Plurisupported Currents on Compact K\"ahler Manifolds}

\author{Nicholas McCleerey}

\begin{document}

\begin{abstract}
Let $X$ be a compact K\"ahler manifold. We study plurisupported currents on $X$, i.e. closed, positive $(1,1)$-currents which are supported on a pluripolar set. In particular, we are able present a technical generalization of Witt-Nystr\"om's proof of the BDPP conjecture on projective manifolds, showing that this conjecture holds on $X$ admitting at least one plurisupported current $T$ such that $[T]$ is K\"ahler.

One of the steps in our proof is to show an upper-bound for the pluripolar mass of certain envelopes of quasi-psh functions when the cohomology class is shifted, a result of independent interest. Using this, we are able to generalize an inequality of McKinnon and Roth to arbitrary pseudoeffective classes on compact K\"ahler manifolds.
\end{abstract}

\maketitle

\section{Introduction}

Suppose that $(X^n, \omega)$ is a closed, compact K\"ahler manifold of complex dimension $n$. It is a well known phenomena that a large number of algebro-geometric questions can be formulated and solved on $X$ by replacing divisors on $X$ with suitable closed positive currents $T$, trading algebraic techniques for analytic ones.

A particular, very explicit example of this is the replacement of a divisor $D$ with its current of integration, $\llbracket D\rrbracket$. From the analytic point of view, $\llbracket D\rrbracket$ is quite singular, being supported on a set of (real) Hausdorff dimension $2n - 2$. In this case, our ability to make statements about or using $\llbracket D\rrbracket$ relies essentially on the underlying (smooth) structure of the supporting set itself -- as is well known, $D$ will be an analytic manifold outside a proper set of (real) Hausdorff dimension $2n - 4$. This allows wide latitude to work with $\llbracket D\rrbracket$ by restricting to the smooth part of $D$, among other things, and results in strong rigidity theorems for currents which charge divisors, e.g. the various support theorems (\cite{Fed, Dem7}) and the Siu Decomposition \cite{Siu1}.

The starting point of the present work is the observation that one can work with a much larger class of objects by requiring only that the support of the current be pluripolar. We make the following definition:

\begin{defn}
We say a closed, positive current $T$, of bidegree $(1,1)$, is {\bf plurisupported} if $T$ is supported on a pluripolar set.
\end{defn}

Recall that, on a closed manifold $X$, we say a set $A$ is pluripolar if there exists a quasi-psh function $\vp$ on $X$ such that $A \subseteq \{\vp = -\infty\}$.

It is an immediate result of Bedford-Taylor theory \cite{BT2, BT3} (see also \cite{BEGZ}, right after (1.2)), Federer's Support theorem \cite[Thm. 4.1.20]{Fed} (see also \cite[Thm. 2.4.2]{King}), and basics of plurisubharmonic functions (c.f. \cite{GZ_book}), that if $T$ is plurisupported and $T = \theta + i\p\pbar\vp$ for some smooth, closed, real $(1,1)$-form $\theta$ with $\vp\in\PSH(X,\theta)$, then $\mathrm{Supp}(T)\subseteq \{\vp = -\infty\}$ and the $(2n-2)$-Hausdorff measure of $\mathrm{Supp}(T)$ is non-zero. If one assumes even moderately stronger conditions than this however, such currents usually end up being currents of integration along divisors:

\begin{prop}[Dinh-Lawrence \cite{DinhLawr}, Theorem 4.6, Corollary 4.7]\label{eyy}
Suppose that $T$ is a closed positive $(1,1)$-current on a closed complex manifold $X$ of complex dimension $\geq 2$. If the support of $T$ has locally finite $(2n-2)$-Hausdorff measure outside of some Stein domain, then $T$ is the current of integration along an effective $\R$-divisor.
\end{prop}

It is known that not all plurisupported currents are currents of integration. For instance, one can construct explicit examples on $\C$ by looking at the Evan's potential of certain generalized Cantor sets, e.g. Example \ref{not_irred} (c.f. \cite{Rans, Car}) -- these can then be pulled back and pushed forward to provide examples on any projective manifold, as detailed in \cite[Section 3]{ThaiVu}, for instance. Another example comes from a well-known construction due to Wermer \cite{Werm}, which produces a plurisupported current on $\CP^2$ whose support has no ``analytic structure" in the unit ball \cite[Example 6.1]{DuvSib} (see also \cite{Lev, LevSlod}).

While the underlying supporting sets of plurisupported currents may be rather extreme, the currents themselves are very regular in a potential theoretic sense, allowing us to generalize many classical properties of divisors to plurisupported currents. In particular, we will make repeated use of the following result, originally due to Demailly; the version we use can be found in \cite[Lemma 4.1]{ACCP}:

\begin{lem}\label{subtraction}
Let $X^n$ be a compact K\"ahler manifold, and suppose that $T, S$ are closed, positive $(1,1)$-currents on $X$. Suppose that $S := \beta + i\p\pbar\psi$ is plurisupported and let $T := \alpha + i\p\pbar\vp$.

If $\vp$ is more singular than $\psi$, then $\vp - \psi$ is an $(\alpha - \beta)$-psh function on $X$.
\end{lem}

Our original preprint provided a simple proof of Lemma \ref{subtraction} using plurifine potential theory -- however, since this lemma was already known with a simple proof not utilizing this theory, we've removed this proof. We thank Duc-Viet Vu for pointing this out to us.

Using Lemma \ref{subtraction}, it makes sense to define a generalized Siu decomposition (which we call the pluripolar Siu decomposition (Defn. \ref{pSiuDecom_curr})), given by decomposing a closed positive current $T$ into its plurisupported and non-pluripolar components, as in \cite{BEGZ}. If $T = \theta_\vp$, then Lemma \ref{subtraction} allows us to extend this to a decomposition of the singularity type $[\vp]$ as well (Defn. \ref{pSiuDecom_singtype}). The most notable difference between the pluripolar and the classical Siu decompositions comes from the fact that a general plurisupported current cannot be decomposed into a sum of irreducible components (Example \ref{not_irred}), even though there exists a natural notion of an irreducible plurisupported current (which ends up agreeing the notion of an extremal current (Defn. \ref{irred})).

Lemma \ref{subtraction} allows us to prove the existence of maxima and minima of plurisupported currents (see subsection \ref{subsect_maxmin}). The existence of minima in particular allows us to define a natural mass norm on the space of differences of plurisupported currents (it is essentially the total variation of the difference), which turns this space into a Banach space (Theorem \ref{Banach}) -- compare this, for instance, to \cite{CegWik}, which uses the Monge-Amp{\`e}re operator to turn the local space $\mathscr{F}(\Omega)$ into a Banach space, and shows that this topology need not be equivalent to the topology induced by the Laplacian (which agrees with the norm studied here when restricted to the set of plurisupport potentials). See also \cite{AhagCegCzyz}, for instance. 

Theorem \ref{Banach} is shown by comparing the metric induced by our mass norm with the $d_\mathcal{S}$ metric of \cite{DDL4}, and showing that they are basically equivalent; in particular, we thus see that, in this case, $d_\mathcal{S}$-convergence provides an analogue for potentials of the convergence in total variation of measures, affirming its apparent position as the strongest known type of convergence for quasi-psh functions; compare this to the results on convergence in capacity in \cite{HaiHiep, Czyz, HawZaw, KarNguTru}.

\begin{thm}\label{converge_intro}
(Theorem \ref{converge}) Suppose that a sequence $\vp_i\in \PSH(X,\omega)$, with $\int_X \langle \omega_{\vp_i}^n\rangle \geq \delta > 0$, $d_\mathcal{S}$-converges to some $\vp\in\PSH(X,\omega)$ as $i\rightarrow\infty$. Then the singular parts of the pluripolar Siu decomposition $\llbracket \omega_{\vp_i}\rrbracket$ converge strongly to $\llbracket \omega_{\vp}\rrbracket$ as $i\rightarrow\infty$.
\end{thm}

See subsection \ref{subsect_Banach} for more specific results.

The results in Section 2 suggest a number of interesting follow up questions. For currents of integration, the classical line of thinking is that the rigidity of the currents stems from the rigidity of the underlying divisor. One would like to know if a similar result holds for plurisupported currents:

\begin{conj}\label{conjecture}
Suppose that $T$ and $S$ are two essentially disjoint (Defn. \ref{essdisj}), plurisupported currents such that $\mathrm{Supp}(T) = \mathrm{Supp}(S)$. Then either $T = 0$ or $S = 0$.
\end{conj}

A variant of this question has been asked before by Slodkowski \cite{Slod}, in relation to extremal currents. There are many examples of extremal currents which are not plurisupported, e.g. many fractal sets from dynamics support extremal currents, and these generally turn out to have continuous potentials.

Our main motivation in proposing Conjecture \ref{conjecture} is actually to study intersections of supports of essentially disjoint plurisupported currents, and probe how similar they are to intersections of divisors -- a clear first step is to rule out the situation in Conjecture \ref{conjecture}. Optimistically, one hopes that such intersections should have Hausdorff dimension $< 2n -2$, which would imply, for instance, the reverse direction of Theorem \ref{converge_intro} -- see Remark \ref{MeStupid}. It does not seem reasonable to expect a similar result for extremal currents.\\
\newline

Aside from studying plurisupported currents on their own, we provide an application of their existence to the well-known conjecture of Boucksom--Demailly--Pa\u{u}n--Peternell \cite[Conj. 10.1]{BDPP} on the duality between the pseudoeffective and movable cones, in a manner analogous to the recent, celebrated paper of Witt-Nystr\"om \cite{WN1} (indeed, upon reading \cite{WN1}, one is naturally lead to wonder if such currents can be substituted for divisors in the main argument). Our main result in this direction is the following:

\begin{thm}\label{kookoo}
(Theorem \ref{crank}) Let $Y$ be a compact K\"ahler manifold, and suppose that $\mu: X\rightarrow Y$ is a proper modification such that $X$ admits a plurisupported current $T$ with $[T]$ a K\"ahler class. Then \cite[Conj. 10.1]{BDPP} is true on $Y$.
\end{thm}

Currently, \cite[Conj. 10.1]{BDPP} is known to hold for projective manifolds, thanks to \cite{WN1}. If $X$ is only K\"ahler, some partial progress has been made by showing special cases of equivalent versions of \cite[Conj. 10.1]{BDPP} (c.f. \cite{BFJ, WN1, Xiao3}); a result of Tosatti \cite{To7} shows that $\langle \alpha^n\rangle = \alpha\cdot\langle\alpha^{n-1}\rangle$ for pseudoeffective classes $[\alpha]$ which are not big, and another, very recent, result of Witt-Nystr\"om shows that $\vol(\alpha + t\beta)$ is differentiable in $t$ whenever $[\beta] = c_1(D)$, for an effective divisor $D$ \cite[Thrm. C]{WN3}. We also mention work on the ``qualitative" part of this conjecture, which is known on even some non-K\"ahler manifolds \cite{Xiao1, Pop, GuLu2}.

Our proof follows the same route as \cite{WN1}, i.e. we first prove a particular case of (the general form of) Demailly's weak transcendental holomorphic Morse inequalities  (c.f. \cite[Prop. 1.1]{Xiao3}) in Theorem \ref{WN_MN}. We then use these to show the asymptotic orthogonality property of approximate Zariski decompositions (see \cite{BDPP} or the appendix to \cite{WN1}). Compared to \cite{WN1}, we are obviously able to handle general plurisupported currents (as opposed to only divisors), but we are also able to deal directly with arbitrary big classes in Theorem \ref{WN_MN}, as opposed to only working with nef classes.

Theorem \ref{kookoo} provides a potentially new approach to solving the BDPP conjecture: showing that an arbitrary K\"ahler manifold admits at least one plurisupported current whose cohomology class is K\"ahler. To the best of the author's knowledge, very little seems to be known in this direction. The currently available methods of construction fundamentally come from working on $\C^n$, and, in light of Proposition \ref{eyy}, it seems difficult to adapt them directly to produce plurisupported currents on non-projective manifolds (again, on projective manifolds, the conjecture is already known thanks to Witt-Nystr\"om \cite{WN1}). This leads to the following question:

\begin{quest}\label{ookook}
Suppose that $T$ is a plurisupported current on a compact K\"ahler manifold $X$. Is $[T] \in \mathrm{NS}(X, \R)$?
\end{quest}

Refuting Question \ref{ookook} would be a major first step towards showing that Theorem \ref{kookoo} represents an honest generalization of the main result of \cite{WN1}. Conversely, an affirmative answer to Question \ref{ookook} would represent an extraordinary generalization of Kodaira's Embedding Theorem, replacing divisors with plurisupported currents. 

Question \ref{ookook} is also related to the question of studying ``complements" of complete K\"ahler domains -- if $T$ is plurisupported and $\mathrm{Supp}(T)$ is closed, then $X\setminus \mathrm{Supp}(T)$ is a complete K\"ahler manifold (c.f. \cite[Ex. 3.2.1]{XuLiu_thesis}).\\

Our proof of Theorem \ref{kookoo} relies on Lemma \ref{subtraction} and a separate generalization of the later half of the proof of \cite{WN1}:

\begin{thm}\label{ppbound_intro}
(Theorem \ref{ppbound}) Let $X^n$ be a compact K\"ahler manifold, and suppose that $[\theta], [\sigma], [\rho]$ are pseudoeffective classes on $X$ such that $\theta = \sigma+\rho$. Let $\xi, \psi\in\PSH(X,\theta)$ and $\eta\in\PSH(X,\rho)$ be such that $P_\sigma(\psi - \eta) \in\PSH(X,\sigma)$ is not $-\infty$, $\psi\leq\xi$ (so that $P_\sigma(\xi - \eta) \not=-\infty$ also), and $\eta$ has small unbounded locus. Then we have the following estimate:
\[
\int_X \langle\sigma^n_{P_\sigma(\xi - \eta)}\rangle - \int_X \langle \sigma_{P_\sigma(\psi - \eta)}^n\rangle \leq \int_X\langle\theta^n_\xi\rangle - \int_X \langle\theta_\psi^n\rangle.
\]
\end{thm}

The proof of Theorem \ref{ppbound_intro} builds on the techniques in \cite{DDL2} (see also \cite{DT}) to generalize the proof in \cite{WN1}. Read literally, Theorem \ref{ppbound_intro} asserts an upper bound on the pluripolar mass of certain envelopes when one shifts cohomology classes. It is a generalization of a similar result from the author's thesis \cite[Theorem 3.2.1]{McC3}.

Using the same idea as in \cite{McC3}, Theorem \ref{ppbound_intro} can be used to generalize a well-known inequality of McKinnon-Roth \cite{MR} (which they attribute to Salberger, in unpublished work) related to Seshadri constants, Theorem \ref{MR}. This inequality has recently found a number of applications in algebraic geometry, particularly in relation to the study of K-stability (c.f. \cite[Thrm. 2.3]{Fuj1} and \cite{Fuj2}), and we hope our generalization might find similar usage in the study of constant scalar curvature metrics, for example. Compared to the versions in \cite{McC3}, Theorems \ref{ppbound_intro} and \ref{MR} apply to general pseudoeffective classes, instead of just K\"ahler classes.\\
\newline

We now briefly outline the rest of the paper. In Section 2 we define the pluripolar Siu decomposition (Defns. \ref{pSiuDecom_curr} and \ref{pSiuDecom_singtype}) and prove basic properties of plurisupported currents. We prove Theorem \ref{ppbound_intro} in Section 3, as well as the claimed generalization of the results in \cite{WN1}, Theorems \ref{WN_MN} and \ref{kookoo} (Theorem \ref{crank}); Theorem \ref{WN_MN} in particular requires the use of Theorem \ref{transfer}, which is a generalization of an observation of Berman \cite{Be3} (Lemma \ref{hoponpop}). Finally, in Section 4 we make an attempt to study the problem of determining existence of plurisupported currents, discussing two different, naive starting points of inquiry.
\\
\newline

We finish this introduction by listing here some of the notation we will use in this paper, especially that which is not standard in the literature. If $\vp: X\rightarrow \R\cup\{-\infty\}$ is a function, we define its singularity type to be:
\[
[\vp] := \{ \psi: X\rightarrow \R\cup\{-\infty\}\ |\ \abs{\vp - \psi} \leq C\}
\]
where here $C$ is a constant that is allowed to depend on $\psi$. In general, we will only work with singularity types of quasi-psh functions, so one should assume that any function $\vp$ is quasi-psh unless otherwise stated. There is a standard ordering on singularity types, given as follows:
\[
[\vp] \leq [\psi] \text{ if and only if } \vp \geq \psi + C.
\]
This makes sense grammatically (as it reads ``$[\vp]$ is less singular than $[\psi]$"), but runs counter to the functional inequality.

If $\theta$ is a smooth, closed, real $(1,1)$-form, we write $[\theta]\in H^{1,1}(X,\R)$ for the corresponding (real) Dolbeault cohomology class. We write $\PSH(X,\theta)$ for the space of $\theta$-psh functions. If $f: X\rightarrow \R$, we write:
\[
P_\theta(f) := \left(\sup\{\vp\in\PSH(X,\theta)\ |\ \vp \leq f\}\right)^*,
\]
where here the star $\cdot^*$ represents upper-semicontinuous regularization. If $f$ is a difference of quasi-psh functions, we can remove the star by Hartog's Lemma, so that $P_\theta(f) \leq f$ -- we will only consider such $f$ in this paper. When $f= 0$, we write $V_\theta:= P_\theta(0)$. It is easy to see that $V_\theta$ realizes the unique minimal singularity type among all $\theta$-psh functions \cite{DPS}.

If $E\subset X$ is a divisor, we write $[E]$ for the first Chern class of $E$ (usually written $c_1(E)$) and $\llbracket E\rrbracket$ for the current of integration along $E$.

If $T_1,\ldots, T_m$ are closed positive currents on $X$, we denote the non-pluripolar product \cite{BEGZ} of these by $\langle T_1\wedge\ldots\wedge T_m\rangle$. If $[\theta]$ is a cohomology class, we denote the movable product of $[\theta]$ \cite{BDPP} by slight abuse of notation as $\langle \theta^k\rangle := \langle \theta_{V_\theta}^k\rangle$, $1\leq k \leq n$.

\subsection*{Acknowledgments:} We would like to thank V. Guedj, for pointing out \cite{Rans} to us, D. Witt-Nystr\"om for discussions about the McKinnon-Roth inequality, and also J. Wiegerinck, for help with an earlier draft. We also thank T. Darvas, V. Tosatti, and D.-V. Vu for many helpful comments and suggestions. Additionally, we would like to thank M. Geis and S. Zelditch for their encouragement and interest in the current work. Parts of this work was carried out while the author was still a graduate student at Northwestern University, paritally supported under the NSF RTG training grant DMS-1502632. 


\section{Plurisupported Currents}


We define the pluripolar Siu decomposition and other show other facts about plurisupported currents, including Theorem \ref{Banach}. Theorem \ref{transfer} is proved in subsection \ref{subsect_Berman}, and will be used crucially later in Section 3. We start by proving a simple corollary of Lemma \ref{subtraction}. Recall that our convention for the partial ordering for two singularity types is opposite to their partial ordering as functions.

\begin{cor}\label{ordering}
Let $X^n$ be a compact K\"ahler manifold, and suppose that $T, S$ are closed, positive $(1,1)$-currents on $X$. Suppose that $S := \beta + i\p\pbar\psi$ is plurisupported and let $T := \alpha + i\p\pbar\vp$. Then:
\[
[\psi] \leq [\vp]\quad\text{ if and only if }\quad S \leq T.
\]
\end{cor}
\begin{proof}
If $\vp \leq \psi + C$, this is just Lemma \ref{subtraction}. The other direction is in fact always true, without any assumption on $\psi$ -- we provide the short proof for completeness. If $S\leq T$, then $[T - S] = [\alpha - \beta]$ is pseudoeffective, and so there exists a $\xi\in\PSH(X, \alpha - \beta)$, $\xi \leq 0$, such that $\alpha - \beta + i\p\pbar\xi = T-S$. It follows that:
\[
\alpha + i\p\pbar(\xi + \psi) = T = \alpha + i\p\pbar \vp
\]
and so $\psi \geq \xi + \psi = \vp - C$, as $X$ is closed.
\end{proof}


\subsection{The Pluripolar Siu Decomposition}


In this subsection, we decompose any closed positive current into a sum of a plurisupported current and a current which is weakly ``regular," in some pluripotential theoretic sense (Proposition \ref{char_regular_component}). This will be the pluripolar Siu decomposition mentioned in the introduction. We then collect some of the more basic, formal properties of this decomposition.

Suppose that $T\in[\theta]$ is a closed positive current, and that $T = \theta + i\p\pbar\vp$. We define the non-pluripolar product of $T$, written $\langle T\rangle$, by:
\[
\langle T\rangle = \chi_{\{\vp > -\infty\}} T.
\]
By \cite{BEGZ}, $\langle T\rangle$ will be closed and positive. It follows that the current:
\[
\llbracket T \rrbracket = T - \langle T\rangle = \chi_{\{\vp=-\infty\}} T
\]
is also closed. It is clearly also positive, and evidently plurisupported. 

\begin{defn}\label{pSiuDecom_curr}
We shall refer to the decomposition:
\[
T := \langle T\rangle + \llbracket T \rrbracket
\] 
as the {\bf pluripolar Siu decomposition of $T$}.
\end{defn}

As justification for the name, note that when $T$ has analytic singularities, the pluripolar Siu decomposition precisely agrees with the usual Siu decomposition (see Proposition \ref{I-sings} for a slightly more general statement). In general, the two decompositions will differ however, as the Siu decomposition will only pick up the divisorial components of $\llbracket T\rrbracket$.

Choose smooth closed forms $\sigma\in[\langle T\rangle]$ and $\rho\in[\llbracket T\rrbracket]$ such that $\theta = \sigma + \rho$. Then we can find potentials $u \in \PSH(X,\sigma)$, $v\in \PSH(X,\rho)$, $u, v \leq 0$, such that:
\[
\langle T\rangle = \sigma + i\p\pbar u\text{ and } \llbracket T\rrbracket = \rho + i\p\pbar v.
\]
It follows that there exists a constant $C$ such that:
\[
\vp = u + v + C,
\]
and so $[\vp] = [u] + [v]$. Changing any of the smooth forms $\sigma$, $\rho$, or $\theta$ will only change $u$ and/or $v$ up to a smooth, bounded term, so the singularity types $[u]$ and $[v]$ are well defined. 

\begin{defn}\label{pSiuDecom_singtype}
We refer to the above splitting of $[\vp]$ as the {\bf pluripolar Siu decomposition of the singularity type $[\vp]$}, and write:
\[
\langle \vp\rangle = [u]\text{ and } \llbracket \vp \rrbracket = [v].
\]
\end{defn}

The following pair of definitions will be convenient, but please note the potential for confusion:

\begin{defn}
Suppose that $[\vp]$ is the singularity type of a quasi-psh function $\vp$. We say that $[\vp]$ is a {\bf plurisupported singularity type} if there exists some $\psi\in [\vp]$ and some smooth, closed, real $(1,1)$ form $\theta$, such that $\theta_\psi$ is a plurisupported current.
\end{defn}

Note that, if $[\vp]$ is plurisupported, then for a general $v \in [\vp]$ with $v\in\PSH(X,\sigma)$, we will not usually have that $\sigma_v$ is plurisupported, even if $\theta_v$ is plurisupported for some other form $\theta$. Also note that for most quasi-psh $v\in[\vp]$, there will not be any corresponding form $\theta$ such that $\theta_v$ is plurisupported at all.

Thus, in order to get a plurisupported current from a plurisupported singularity type, one must specify both a specific quasi-psh function $\psi$ and a specific smooth form $\theta$ simultaneously (though there will always be many such choices). We call such a pair $(\psi, \theta)$ a {\bf plurisupported representative of $[\vp]$}, implicitly referring to the current $\theta_\psi$. Note that this current will always be unique, by Corollary \ref{ordering}.\\

It is easy to see from the definitions that $\llbracket\vp\rrbracket$ is plurisupported for any singularity type $[\vp]$ with $\vp$ quasi-psh.\\

While the pluripolar Siu decomposition of $[\vp]$ as currently defined does not depend on our specific choices of potentials for the current $T$, it superficially appears to depend on $T$ itself. The following alternate characterization of $\llbracket \vp \rrbracket$ shows this is not the case.

\begin{prop}\label{char_singular_part}
Fix a closed positive current $T = \theta + i\p\pbar \vp$. We have the following characterizations of $\llbracket T\rrbracket$ and $\llbracket \vp\rrbracket$:
\begin{enumerate}
\item $\llbracket T\rrbracket$ is the largest plurisupported current such that $\llbracket T\rrbracket \leq T$.
\item $\llbracket \vp\rrbracket$ is the largest plurisupported singularity type such that $\llbracket \vp \rrbracket \leq [\vp]$.
\end{enumerate}
\end{prop}
\begin{proof}
$(1)$ is clear-- suppose $T = R + \llbracket T'\rrbracket$ is another decomposition such that $\llbracket T'\rrbracket$ is plurisupported and $R \geq 0$. Let $\vp$ be a potential for $T$ and $\psi$ a potential for $\llbracket T'\rrbracket$. By Corollary \ref{ordering}, $[\psi] \leq [\vp]$, so that $\{\psi = -\infty\} \subseteq \{\vp = -\infty\}$. Since $\psi$ is plurisupported, $\llbracket T'\rrbracket$ is supported on $\{\vp = -\infty\}$, and hence:
\[
\llbracket T'\rrbracket = \chi_{\{\vp = -\infty\}} \llbracket T'\rrbracket \leq \chi_{\{\vp = -\infty\}} T = \llbracket T\rrbracket.
\]

We now show $(2)$. Suppose that $[\xi] \leq [\vp]$ is plurisupported. Let $\gamma_\xi$ be a plurisupported representative of $[\xi]$. By Corollary \ref{ordering}, we have that $\gamma_\xi \leq T$, and so by $(1)$, $\gamma_\xi \leq \llbracket T\rrbracket$. Applying Corollary \ref{ordering} again shows that $[\xi] \leq \llbracket \vp\rrbracket$. Since $\llbracket \vp\rrbracket$ is plurisupported and $\llbracket \vp \rrbracket \leq [\vp]$ by definition, we are done.
\end{proof}

We see that $(2)$ is defined only in terms of the singularity type $[\vp]$ -- thus, it follows that if $\vp, \vp' \in[\vp]$ are two quasi-psh functions such that $\vp\in\PSH(X,\theta)$ and $\vp'\in\PSH(X,\theta')$ for some smooth, closed, real $(1,1)$-forms $\theta$ and $\theta'$, then we do in fact have $\llbracket \vp\rrbracket = \llbracket \vp'\rrbracket$, independent of the two currents $\theta_\vp$ and $\theta'_{\vp'}$.

We finish this subsection by collecting some basic properties of the pluripolar Siu decomposition of $[\vp]$ -- all are quite elementary, given Lemma \ref{subtraction}, and will be familiar to most readers. To minimize redundancy, we skip the equivalent statements for the pluripolar Siu decomposition of a closed positive current; one can recover these by choosing appropriate representatives in the pluripolar Siu decomposition of $[\vp]$. All statements can be carried over in this manner; so for instance, if $T = \theta + i\p\pbar\vp$ and $\langle \vp\rangle = [0]$, then this is the same as saying $\langle T\rangle$ has a bounded potential. The one slight point of note is that $\llbracket \vp \rrbracket = [0]$ implies that $\llbracket T\rrbracket \equiv 0$, by Bedford and Taylor \cite{BT2, BT3} (c.f. the introduction, right after Definition 1.1).

\begin{prop}\label{projection}
Suppose that $[\vp]$ is the singularity type of a quasi-psh function $\vp$. Then $\langle \langle \vp\rangle \rangle = \langle \vp\rangle$ and $\llbracket \langle \vp \rangle\rrbracket = 0$. Similarly $\langle\llbracket \vp \rrbracket\rangle = 0$ and $\llbracket \llbracket \vp\rrbracket \rrbracket = \llbracket \vp\rrbracket$.
\end{prop}
\begin{proof}
Consider the pluripolar Siu decomposition of $\langle \vp \rangle$; expanding it out gives:
\[
[\vp] = \llbracket \vp \rrbracket + \llbracket \langle \vp\rangle \rrbracket +\langle \langle \vp\rangle \rangle.
\]
By $(2)$ of Proposition \ref{char_singular_part}, $\llbracket \vp\rrbracket + \llbracket \langle \vp \rangle \rrbracket \leq \llbracket \vp\rrbracket$, implying that $\llbracket \langle\vp\rangle \rrbracket = 0$ and hence $\langle \langle \vp\rangle \rangle = \langle \vp\rangle$.

Similarly, $\llbracket \llbracket \vp\rrbracket \rrbracket = \llbracket \vp \rrbracket$ by $(2)$ of Proposition \ref{char_singular_part}, as $\llbracket \vp\rrbracket$ is plurisupported and obviously $\llbracket \vp \rrbracket \leq \llbracket \vp\rrbracket$ and $\llbracket\llbracket\vp\rrbracket\rrbracket \leq \llbracket\vp\rrbracket$.
\end{proof}

\begin{prop}\label{char_regular_component}
Fix a singularity type $[\vp]$ of a quasi-psh function $\vp$. We have that $v \in \mathcal{E}_1(X, \omega)$ for each $v\in \langle \vp\rangle$ and each K\"ahler form $\omega$ such that $v\in\PSH(X,\omega)$. Here $\mathcal{E}_1(X,\omega) := \{\psi\in\PSH(X,\omega)\ |\ \int_X \langle \omega_\psi\wedge\omega^{n-1}\rangle = \int_X \omega^n\}$.
\end{prop}
\begin{proof}
We claim this property is equivalent to $\llbracket \langle \vp\rangle\rrbracket = 0$. Suppose that $v\in\langle \vp\rangle$ and $\omega$ is a K\"ahler form with $v\in \PSH(X,\omega)$. We need to check that $v\in \mathcal{E}_1(X,\omega)$. Note that:
\[
\int_X (\omega^n - \langle\omega_v\rangle\wedge\omega^{n-1}) = \int_X \llbracket \omega_v\rrbracket \wedge \omega^{n-1}
\]
so this is equivalent to just checking that $\llbracket \omega_v\rrbracket = 0$, as claimed.  By Proposition \ref{projection}, this is immediate.
\end{proof}

\begin{rmk}
The class $\mathcal{E}_1(X,\omega)$ has appeared before in the pluripotential literature, in a paper of Coman-Guedj-Zeriahi \cite{CGZ} (they refer to it as $\mathcal{E}^1(\omega,\omega)$, thinking of it as a kind of ``twisted" finite energy class). In their paper, they ask if it is possible to determine a relation between $\mathcal{E}_1(X,\omega)$ and the domain of definition of the complex Monge-Amp{\`e}re operator on a compact K\"ahler surface; in its most optimistic formulation, they ask if these spaces are in fact equal \cite[Question 5.3]{CGZ}. 

We believe this optimistic take to be true, though it would appear that almost no progress has been made in addressing this question since \cite{CGZ} (to be clear, the question is very subtle, given the known differences between the global domain of definition of the complex Monge-Amp{\`e}re operator, and the local domain of definition -- see \cite[Section 3]{CGZ} and B\l ocki \cite{Bl2, Bl3}). If true, this would provide a stronger characterization of the regular part of the pluripolar Siu decomposition (at least in dimension 2), exactly in line with what is known for the regular part of the Siu decomposition of a current with analytic singularities, by a well-known theorem of Demailly \cite[Thrm. 4.5, pg. 152]{Dem7}.

As a matter of terminology, we remark that the class $\mathcal{E}_1(X,\omega)$ is the set of $\omega$-psh functions which have ``full-mass" when considered as $\omega$-subharmonic functions. We make no use of this fact, but it explains our deviation in notation.
\end{rmk}

\begin{prop}\label{sing_ordering}
The singular part of the pluripolar Siu decomposition is order preserving; if $[\psi] \leq [\vp]$ for two quasi-psh functions $\vp$ and $\psi$, then:
\[
\llbracket \psi \rrbracket \leq \llbracket \vp \rrbracket.
\]
\end{prop}
\begin{proof}
We have:
\[
\llbracket \psi\rrbracket \leq \langle \vp\rangle + \llbracket \vp \rrbracket.
\]
Pick quasi-psh representatives of each singularity type above, denoted $\xi, u,$ and $v$, respectively, and smooth forms, denoted $\theta, \sigma,$ and $\rho$, respectively, such that $\theta_\xi$ and $\rho_v$ are plurisupported, $\sigma_u \geq 0$, and $u + v \leq \xi$. 

By Corollary \ref{ordering}, we have 
\[
\theta_\xi \leq \sigma_u + \rho_v.
\]
Let $A = \{\xi = -\infty\}\cup\{v = -\infty\}$. This is a complete pluripolar set, and since:
\[
\sigma_u = \langle \sigma_u\rangle,
\]
we have $\chi_A \sigma_u = 0$, by \cite{BEGZ}. On the other hand $\chi_A\theta_\xi = \theta_\xi$ and $\chi_A\rho_v = \rho_v$, so that actually:
\[
\theta_\xi \leq \rho_v.
\]
Corollary \ref{ordering} now implies $[\xi] \leq [v]$, as desired.
\end{proof}

\begin{prop}\label{sing_add}
The pluripolar Siu decomposition is additive; if $[\vp]$ and $[\psi]$ are singularity types of two quasi-psh functions, then:
\[
\langle \vp + \psi\rangle = \langle \vp\rangle + \langle \psi\rangle\quad \text{ and }\quad \llbracket \vp + \psi\rrbracket = \llbracket \vp\rrbracket + \llbracket \psi\rrbracket.
\]
\end{prop}
\begin{proof}
It is clear from $(2)$ of Proposition \ref{char_singular_part} that:
\[
\llbracket \vp \rrbracket + \llbracket \psi\rrbracket \leq \llbracket \vp+ \psi\rrbracket.
\]
Let $[\xi] = \llbracket \vp+ \psi\rrbracket - \llbracket \vp \rrbracket -  \llbracket\psi\rrbracket$ denote the difference; by Lemma \ref{subtraction} it will contain a quasi-psh representative and be plurisupported, so that $[\xi] = \llbracket\xi\rrbracket$. We have:
\[
(\llbracket \vp + \psi\rrbracket - [\xi]) + \langle \vp\rangle + \langle \psi\rangle = [\vp] + [\psi] = [\vp + \psi] = \llbracket \vp + \psi\rrbracket + \langle \vp + \psi\rangle,
\]
so that
\[
\langle \vp \rangle + \langle\psi\rangle = \langle\vp + \psi\rangle + [\xi].
\]
By Propositions \ref{projection}, \ref{char_regular_component}, and \ref{sing_ordering} we have:
\[
[\xi] = \llbracket \xi\rrbracket \leq \llbracket \langle \vp \rangle + \langle \psi \rangle\rrbracket = 0,
\]
so we are done.
\end{proof}


\subsection{Structure of Plurisupported Currents}\label{subsect_maxmin}


We show that one can take maxima and minima of plurisupported currents. We use these to define what if means for a plurisupported current to be irreducible, which turns out to be equivalent to it being extremal. Then, we provide an example of a plurisupported current which cannot be decomposed as an infinite sum of extremal plurisupported currents.

Suppose $[\vp], [\psi]$ are plurisupported singularity types and $\theta_\vp$ and $\sigma_\psi$ are two plurisupported representatives. It is easy to see that the singularity type $[\max\{\vp, \psi\}]$ is well-defined, and that it is the largest singularity type such that:
\[
[\max\{\vp, \psi\}] \leq [\vp]\quad \text{ and }\quad [\max\{\vp, \psi\}] \leq [\psi].
\]
Suppose that $[v]$ is plurisupported and $[v] \leq [\max\{\vp, \psi\}]$. Then by Proposition \ref{sing_ordering} we have that:
\[
\llbracket v\rrbracket = [v] \leq \llbracket \max\{\vp, \psi\} \rrbracket \leq \llbracket \vp\rrbracket = [\vp].
\]
Thus, $\llbracket \max\{\vp, \psi\}\rrbracket$ is the largest plurisupported singularity type less than both $[\vp]$ and $[\psi]$. It is also clear from Corollary \ref{ordering} that if $\rho_\xi$ is a plurisupported representative of $\llbracket \max\{\vp, \psi\}\rrbracket$, then $\rho_\xi$ is the largest plurisupported current such that $\rho_\xi \leq \theta_\vp$ and $\rho_\xi \leq \sigma_\psi$.

\begin{rmk}\label{minimum_remark}
It might be more appropriate to denote $\llbracket \max\{\vp, \psi\}\rrbracket$ by something like $\min\{[\vp], [\psi]\}$, but in the interest of not adding more notation, we have chosen not to do this. Note however that the existence of a minimal current is quite special --  in general the minimum of two currents is not a current at all (it usually fails to be additive, e.g. the minimum of two delta functions).

Also, note that we almost always have that $[\max\{\vp, \psi\}] > \llbracket \max\{\vp, \psi\}\rrbracket$. Consider for instance the case when $\theta_\vp$ and $\sigma_\psi$ are the currents of integration along two irreducible divisors -- generically, $[\max\{\vp, \psi\}]$ will have log poles at exactly a finite number of isolated points, while $\llbracket \max\{\vp, \psi\}\rrbracket$ will be $[0]$.
\end{rmk}

We can also define a maximum of two plurisupported currents by considering $[P_\omega(\min\{\vp, \psi\})]$ for some sufficiently large K\"ahler form $\omega$. We first need the following lemma, which is essentially due to Berman (modulo our usage of Lemma \ref{subtraction}), and will be generalized by Theorem \ref{transfer}:
\begin{lem}\label{hoponpop}
Suppose that $\psi\in \PSH(X,\sigma)$ is plurisupported and that $f \leq C < \infty$ is a difference of quasi-psh functions. Suppose additionally that the envelope:
\[
P_\theta(f + \psi) = \sup\{u\in\PSH(X,\theta)\ |\ u\leq f + \psi\}
\]
is not $-\infty$. Then we have:
\[
P_{\theta}(f + \psi) = \psi + P_{\theta - \sigma}(f).
\]
\end{lem}
\begin{proof}
Suppose that $u\in\PSH(X,\theta)$ is such that $u \leq f + \psi \leq C + \psi$. By Lemma \ref{subtraction}, we have that:
\[
u - \psi \in \PSH(X, \theta - \sigma),
\]
and $u - \psi \leq f$ by Hartog's Lemma. Thus, $u - \psi \leq P_{\theta - \sigma}(f)$, so taking the supremum over all such $u$ establishes one inequality.

For the other inequality, let $v \in \PSH(X, \theta - \sigma)$ be such that $v \leq f$. Then $v + \psi \in \PSH(X, \theta)$ and $v + \psi \leq f + \psi$, so:
\[
v + \psi \leq P_\theta(f+\psi).
\]
Taking the supremum over all such $v$ concludes the lemma.
\end{proof}

\begin{prop}\label{maximum}
Suppose that $[\vp], [\psi]$ are plurisupported, and let $[\xi] = \llbracket\max\{\vp, \psi\}\rrbracket$. Let $\theta_\vp, \sigma_\psi,$ and $\rho_\xi$ be plurisupported representatives of their respective singularity types. Then:
\[
P_{\theta + \sigma - \rho}(\min\{\vp, \psi\}) = \psi + P_{\theta - \rho}(\min\{0, \vp - \psi\}) = \psi + P_{\theta - \rho}(\vp -\psi) + C= \vp + \psi - \xi + C,
\]
where here $C$ is a constant whose exact value might change in each equality. If $\omega \geq \theta + \sigma - \rho$ is K\"ahler, then the above still holds up to taking singularity types:
\begin{equation}\label{formuler}
[P_{\omega}(\min\{\vp, \psi\})] = \llbracket \vp + \psi - \xi\rrbracket = [\vp + \psi - \xi].
\end{equation}
\end{prop}
\begin{proof}
Throughout the proof, we allow the additive constant $C$ to change line-by-line, to minimize notation. Note that $P_{\theta - \rho} (\vp - \psi)$ is not $-\infty$, as $\vp - \xi \in \PSH(X,\theta - \rho)$ by Lemma \ref{subtraction}. 

We will go in reverse order, and start by showing the last equality:
\[
P_{\theta - \rho} (\vp - \psi) = \vp - \xi - C.
\]
Since $\psi \leq \xi + C$, we have:
\[
P_{\theta - \rho}(\vp - \psi) \geq P_{\theta - \rho}(\vp - \xi - C) = \vp - \xi - C,
\]
by Lemma \ref{subtraction}. Let:
\[
P_{\theta - \rho} (\vp - \psi) = \eta + \zeta,
\]
where $\eta\in\llbracket P_{\theta - \rho}(\vp - \psi)\rrbracket$ is plurisupported and $\zeta\in\langle P_{\theta - \rho}(\vp - \psi)\rangle$. We have:
\begin{equation}\label{poopy2}
\vp - \xi - C \leq \eta + \zeta \leq \vp - \psi.
\end{equation}
It follows from the first inequality in \eqref{poopy2} that:
\[
\vp \leq \eta + C,
\]
so that $\vp - \eta$ is well-defined. Since $[\vp]$ is plurisupported, so too is $[\vp - \eta]$. Rearraigning \eqref{poopy2} then gives:
\[
\psi + \zeta \leq \vp - \eta \leq \xi + \zeta + C
\]
so that:
\[
[\psi + \zeta] \geq [\vp - \eta] \geq [\xi + \zeta].
\]
Using Propositions \ref{sing_ordering}, \ref{sing_add}, and \ref{projection}, the above implies that actually:
\[
[\psi] \geq [\vp-\eta] \geq [\xi];
\]
this follows from taking $\llbracket \cdot\rrbracket$ of each term, and then noting that $\llbracket \psi + \zeta\rrbracket = \llbracket \psi\rrbracket + \llbracket\zeta\rrbracket = [\psi]$, and similarly for $[\xi + \zeta]$.

Since we also clearly have $[\vp - \eta] \leq [\vp]$, we thus have that $[\vp - \eta] = [\xi]$ by the definition of $[\xi] = \llbracket\max\{\vp, \psi\}\rrbracket$. Hence we have that $\eta \leq \vp - \xi + C$. But:
\[
\vp - \xi - C \leq \eta + \zeta \leq \vp - \xi + C,
\]
as $\zeta \leq C$, and so we see $[\zeta] = [0]$ and 
\begin{equation}\label{terminate}
[\eta] = [P_{\theta-\rho}(\vp - \psi)] = [\vp - \xi].
\end{equation}
To conclude the last equality, define:
\[
u := P_{\theta - \rho}(\vp - \psi).
\]
Then by Lemma \ref{hoponpop},
\[
u = P_{\theta - \rho}(u) = P_{0}(u - \vp + \xi) + \vp - \xi = C + \vp - \xi,
\]
as needed.

We now show the second equality, namely that:
\[
P_{\theta-\rho}(\min\{0, \vp - \psi\}) = P_{\theta-\rho}(\vp - \psi) - C.
\]
It is clear that there is some $C \geq 0$ such that $P_{\theta-\rho}(\vp - \psi) - C \leq \min\{0, \vp-\psi\}$. In the other direction, note that definitionally $P_{\theta-\rho}(\min\{0, \vp - \psi\}) \leq \vp - \psi$, and so:
\[
P_{\theta-\rho}(\min\{0, \vp - \psi\}) \leq P_{\theta-\rho}(\vp - \psi),
\]
showing that $[P_{\theta-\rho}(\min\{0, \vp - \psi\})] = [P_{\theta-\rho}(\vp - \psi)] = [\vp - \xi]$. The same argument as in the last paragraph establishes the second inequality.

The first equality:
\[
P_{\theta + \sigma - \rho}(\min\{\vp, \psi\}) = \psi + P_{\theta - \rho}(\min\{0, \vp - \psi\})
\]
is just a direct application of Lemma \ref{hoponpop}, as $\min\{0, \vp-\psi\} = \vp - \psi - \max\{0, \vp-\psi\} \leq 0$.\\

To see \eqref{formuler}, note that since $\omega \geq \theta + \sigma - \rho$, we have that $\vp + \psi - \xi \in \PSH(X,\omega)$, and this is enough to repeat all of the above arguments except the one immediately after \eqref{terminate}, which converts equality of the singularity types into pointwise equality. This finishes the proof.

\end{proof}

\begin{rmk}
Equation \eqref{formuler} is the analogue of the familiar formula for functions:
\[
f + g = \max\{f, g\} + \min\{f, g\}.
\]
Given that we need to project $\min\{\vp, \psi\}$ onto the space of PSH functions in \eqref{formuler}, the fact that equality holds is perhaps surprising; compare for instance with the recent inequality \cite[Theorem. 5.4]{DDL4}, which is in general strict \cite[Remark 5.5]{DDL4}. The difference between the two situations is that our minimum $\llbracket\max\{\vp, \psi\}\rrbracket$ is also actually a projection, as noted in Remark \ref{minimum_remark}
\end{rmk}

Subtracting $\xi$ from both sides of Proposition \ref{maximum} gives us:
\begin{equation}\label{hereforafriend}
P_{\theta + \sigma - 2\rho}(\min\{\vp - \xi, \psi - \xi\}) = \vp + \psi - 2\xi,
\end{equation}
and implies that $\llbracket \max\{\vp - \xi, \psi - \xi\}\rrbracket = [0]$. For two random quasi-psh functions, one generally expects $u+v$ to be significantly more singular than $P(\min\{u, v\})$; this suggests that $\vp - \xi$ and $\psi - \xi$ are in some sense ``disjoint." This motivates the following definition:

\begin{defn}\label{essdisj}
We say two plurisupported singularity types $[\vp]$ and $[\psi]$ are {\bf essentially disjoint} if $\llbracket \max\{\vp, \psi\}\rrbracket = 0$.
\end{defn}

The above notion is relative, and one would like an absolute version.

\begin{defn}\label{irred}
We say that a plurisupported singularity type $[\vp]$ is {\bf extremal} if for any other plurisupported singularity type $[\psi]$, we have that $\llbracket \max\{\vp, \psi\}\rrbracket = c[\vp]$ for some $0 \leq c\leq 1$.
\end{defn}

It is easy to see that this is a simple reformulation of the usual definition of an extremal current ($T = T_1 + T_2\implies T_1 = cT_2$) in our setting.

\begin{rmk}
The above constant $c$ could probably be computed as some kind of ``relative type,"  similar to the work of Rashkovskii \cite{Rash}. This would likely be the appropriate replacement for the Lelong number in this very general setting.
\end{rmk}

The canonical example of an extremal plurisupported current is the current of integration along a divisor, by Federer's Support Theorem. One might hope that, similar to the case with divisors, one could decompose any plurisupported current into a countable sum of extremal, plurisupported currents. However, this is not the case, as the following one-dimensional example shows:

\begin{exmp}\label{not_irred}
As is well-known, sufficiently small generalized Cantor sets in $\C$ are polar sets \cite{Rans} (see also \cite{Car}). Fix such a Cantor set $E$ contained in the interval $[0, 1]$. By Evan's Theorem \cite{Rans}, there exists a measure $\mu$ supported on $E$, such that the Newton potential $p_\mu$ of $\mu$ is a subharmonic function on $\C$, harmonic on $\C\setminus E$, and $\equiv -\infty$ on $E$ (see the recent nice paper of Li \cite[Prop. 3.3]{Li2} for a complete, elementary proof of these facts). 

Let $E_{1, j}, \ldots, E_{2^j, j}$, $j \in \{0, 1, \ldots\}$ be the $2^j$ intervals in the $j^\text{th}$ step in the construction of $E$. For any $j$, we then have:
\[
\mu = \sum_{i=1}^{2^j}\chi_{E_{i,j}}\mu,
\]
and hence:
\[
p_\mu = \sum_{i=1}^{2^j} p_{\chi_{E_{i, j}}\mu}.
\]
Clearly each $p_{\chi_{E_{i,j}}\mu}$ is plurisupported, and equally clearly none of them are extremal, each being a sum of two smaller plurisupported currents with disjoint support. 

We now easily see that $[p_\mu]$ cannot be decomposed as a sum of extremal plurisupported singularity types -- suppose it could be:
\[
[p_\mu] = \sum_{k=1}^\infty [\vp_k].
\]
Then for any fixed $k$ and $j$, by definition we would have that $\llbracket \max\{\vp_k, p_{\chi_{E_{i,j}}\mu}\}\rrbracket = c_{i,j}[\vp_k]$ for some constant $0 \leq c_{i,j}\leq 1$. Since the $E_{i,j}$ are disjoint and cover $E$, we must have that there exists exactly one $i_{j,k}$ such that $c_{i_{j,k},j} = 1$ and all the other $c_{i,j} = 0$. Setting $E_j := E_{i_{j,k}, j}$, we have $[\vp_k] \leq [p_{\chi_{E_j}\mu}]$. This holds for any $j$, and so if $T_k$ is a plurisupported representative of $[\vp_k]$, we see that:
\[
\mathrm{Supp}(T_k) \subseteq \cap_{j} E_{j}.
\]
But this intersection can contain at most a single point, so that we must have $T_k = 0$ by the support theorem, a contradiction.
\end{exmp}

One can pull this current back to any appropriate product manifold to produce an example in higher dimensions.


\subsection{Mass Transfer via Plurisupported Currents}\label{subsect_Berman}


We will now use the pluripolar Siu decomposition to prove Theorem \ref{transfer}, which is an essentially optimal version of an observation originally due to Berman \cite{Be3}.

\begin{thm}\label{transfer}
Suppose that $\psi\in \PSH(X,\sigma)$ is plurisupported, $\eta$ is quasi-psh, and $f \leq C < \infty$ is a difference of quasi-psh functions. Suppose additionally that the envelope:
\[
P_\theta(f + \psi - \eta) = \sup\{u\in\PSH(X,\theta)\ |\ u\leq f + \psi - \eta\}
\]
is not $-\infty$. Let $\xi\in\PSH(X,\rho)$ be a plurisupported representative of $\llbracket \max\{\psi, \eta\}\rrbracket$. Then we have:
\[
P_{\theta}(f + \psi - \eta) = \psi - \xi + P_{\theta - \sigma + \rho}(f - \eta +\xi).
\]
In particular, if $[\sigma]$ is big, $\eta = V_\sigma$, and $f$ is quasi-psh with $[f] \leq [V_{\theta - \sigma + \rho}]$, then we have:
\begin{equation}\label{mass}
\int_X \langle (\theta + i\p\pbar P_\theta(f + \psi - V_\sigma))^n\rangle = \vol(\theta - P(\sigma)),
\end{equation}
where here $P(\sigma)$ is the positive part of the divisorial Zariski decomposition \cite{Bo3}.
\end{thm}
\begin{proof}
Let $u\in \PSH(X,\theta)$ be such that:
\[
u \leq f + \psi - \eta \leq C+ (\psi - \xi) - (\eta - \xi).
\]
We wish to show that $\vp := u - \psi + \xi \leq C$; if this is true, one can simply copy the proof of Proposition \ref{hoponpop}.

Since $u$ is bounded from above, we have that $\vp \leq C -\psi + \xi$. Combining this with the above expression gives that $\vp \leq C - \max\{\eta - \xi, \psi - \xi\}$. Replacing $\vp$ and rearranging gives:
\[
u + \max\{\eta - \xi, \psi - \xi\} \leq C + \psi - \xi,
\] 
or
\begin{equation}\label{placeholder}
[u + \max\{\eta - \xi, \psi - \xi\}] \geq [\psi - \xi].
\end{equation}
By \eqref{hereforafriend} (located immediately after Proposition \ref{maximum}), we have that $\llbracket \max\{\eta - \xi, \psi - \xi\}\rrbracket = [0]$ (stictly speaking, one needs to replace $\eta$ with a quasi-psh representative of $\llbracket \eta\rrbracket$ -- we skip the details). We now conclude in a manner similar to the proof of Proposition \ref{maximum}; namely, taking $\llbracket \cdot \rrbracket$ of \eqref{placeholder} gives:
\[
[u] \geq \llbracket u \rrbracket \geq [\psi - \xi],
\]
and so we are done by Proposition \ref{ordering}.

To see \eqref{mass}, recall that $\llbracket V_\sigma\rrbracket$ produces exactly the negative part of the divisorial Zariski decomposition of $\sigma$ (Lemma \ref{mnef}), and that $V_{\theta - \sigma + \rho}$ has minimal singularities in $[\theta - \sigma + \rho]$. Thus, if $[f] \leq [V_{\theta - \sigma + \rho}]$, we have:
\[
f - V_\sigma + \xi + \sup_X (V_\sigma - \xi)\geq f \geq V_{\theta - \sigma + \rho} - C
\]
and it follows that $[P_{\theta - \sigma + \rho}(f -V_\sigma + \xi)] = [V_{\theta - \sigma + \rho}]$. \eqref{mass} now follows from \cite{BEGZ}.
\end{proof}

\begin{lem}\label{mnef}
Suppose that $[\sigma]$ is a big class. Then $\llbracket V_\sigma\rrbracket = N(\sigma)$, the negative part of the divisorial Zariski decomposition.
\end{lem}
\begin{proof}
By Boucksom \cite{Bo3}, $V_\sigma$ is bounded on $X\setminus E_{nK}(\sigma)$, the ample locus of $[\sigma]$. Hence, $\llbracket V_\sigma\rrbracket$ is supported on $\{V_\sigma = -\infty\}\subseteq E_{nK}(\sigma)$, which we recall is an analytic set; thus, by Federer's support theorem, $\llbracket V_\sigma \rrbracket$ must be divisorial. It follows that $\llbracket V_\sigma\rrbracket$ is supported on those divisors in $E_{nK}(\sigma)$ along which $V_\sigma$ has positive, generic Lelong number -- this exactly characterizes the negative part of the divisorial Zariski decomposition, by \cite[Prop. 3.6, Defn. 3.7]{Bo3}.
\end{proof}


\subsection{Continuity of the Pluripolar Siu Decomposition}\label{subsect_Banach}


We now study the continuity of the singular part of the pluripolar Siu decomposition. The main result of this subsection is that, for plurisupported currents, $d_\mathcal{S}$-convergence in the sense of Darvas-Di Nezza-Lu \cite{DDL4} can be detected by an integral norm, which turns the space of (differences of) plurisupported currents into a Banach space, thanks to the completeness result \cite[Theorem 4.9]{DDL4}. Moreover, we show that for sequences of plurisupported singularity types, $d_\mathcal{S}$-convergence implies strong convergence of the associated plurisupported currents. This sheds some light on the geometric nature of the $d_\mathcal{S}$ metric, at least in this special case.\\

We begin with an obvious preparatory lemma, regarding finding minimal representatives for a difference of plurisupported currents.

\begin{lem}\label{convenient}
Suppose that $[\vp], [\psi]$ are plurisupported singularity types. Then there exists a unique essentially disjoint pair of plurisupported singularity types $[u]$ and $[v]$ such that:
\[
[u-v] = [\vp - \psi].
\]
Moreover, for any other pair of plurisupported singularity types $[\eta], [\xi]$ with $[\eta - \xi] = [\vp - \psi]$, we have
\[
[u] \leq [\eta]\ \ \ \text{ and }\ \ \ [v] \leq [\xi].
\]
\end{lem}
\begin{proof}
Similar to subsection \ref{subsect_maxmin}, $[u]$ and $[v]$ can be given explicitly. Let $\gamma\in\llbracket \max\{\vp, \psi\}\rrbracket$. Then:
\[
[u] = [\vp - \gamma]\text{ and }[v] = [\psi - \gamma].
\]
It is clear from the definitions that the properties in the lemma hold.
\end{proof}

\begin{defn}
Denote the space of plurisupported singularity types by $\mathcal{T}$ and the space of plurisupported currents by $\mathcal{T}'$. Denote the space of differences of plurisupported singularitiy types by $\delta\mathcal{T}$, and similarly for $\delta \mathcal{T}'$.

We define the following norm on $\delta\mathcal{T}$. Fix a K\"ahler metric $\omega$ on $X$. Let $[\vp - \psi]\in\delta\mathcal{T}$. Let $\theta_\vp$, $\sigma_\psi$, $\rho_\xi$, and $\gamma_\eta$ be plurisupported representatives of $[\vp], [\psi], \llbracket \max\{\vp, \psi\}\rrbracket$, and $[P_{\theta + \sigma - \rho}(\min\{\vp, \psi\})]$ respectively. Then
\begin{equation}\label{norman}
\norm{\vp - \psi}_\omega = \norm{\vp - \psi} := \int_X (\gamma_\eta - \rho_\xi)\wedge\omega^{n-1} = \int_X (\theta_\vp + \sigma_\psi - 2\rho_\xi)\wedge\omega^{n-1}.
\end{equation}
We will also denote the corresponding norm on $\delta\mathcal{T}'$ by $\norm{\cdot}_\omega$, and note the two spaces are obviously isomorphic as normed vector spaces.

If $T - S\in \delta\mathcal{T}'$, it is easy to see that $\norm{T - S}_\omega$ is just the mass of the total variation of the measure $(T - S)\wedge\omega^{n-1}$.
\end{defn}

If $[\vp]$ and $[\psi]$ are essentially disjoint, then \eqref{norman} simplifies to:
\[
\norm{\vp - \psi}_\omega = \int_X \gamma_\eta\wedge\omega^{n-1} = \int_X (\theta_\vp  + \sigma_\psi)\wedge\omega^{n-1},
\]
and if $[\vp] \geq [\psi]$ or $[\vp] \leq [\psi]$, then:
\[
\norm{\vp - \psi}_\omega = \abs{ \int_X (\theta_\vp - \sigma_\psi)\wedge\omega^{n-1}}.
\]

It is easy to see using Lemma \ref{convenient} and Proposition \ref{ordering} that $\norm{\vp - \psi}$ is an actual norm.

Changing $\omega$ obviously produces an equivalent norm, so that the topology on $\delta\mathcal{T}$ is independent of $\omega$. It is easy to see that it is at least as strong as the usual strong topology on the space of currents:

\begin{prop}\label{strong}
Suppose that $[\vp_i - \psi_i], [\vp - \psi] \in \delta\mathcal{T}$ are such that each of the pairs $[\vp_i], [\psi_i]$ and $[\vp], [\psi]$ are essentially disjoint, and 
\[
\norm{(\vp_i - \psi_i) - (\vp - \psi)}_\omega \rightarrow 0.
\]
Let $T_i, S_i, T, S$ be plurisupported representatives of $[\vp_i], [\psi_i], [\vp]$, and $[\psi]$, respectively. Then $T_i - S_i$ strongly converges to $T - S$.
\end{prop}
\begin{proof}

We claim that:
\[
\norm{(\vp_i - \psi_i) - (\vp - \psi)} = \norm{\vp_i - \vp} + \norm{\psi_i - \psi},
\]
so that the positive and negative parts can be handled separately.

Assume the claim for now, and let $0 \leq f \leq 1$ be measurable and $\alpha$ be a smooth, positive $(n-1, n-1)$-form, and choose $\omega$ such that $\alpha \leq \omega^{n-1}$. Let $\ov{T}_i$ be a plurisupported representative of $\llbracket\max\{T_i, T\}\rrbracket$. Then:
\[
\abs{\int_X f(T_i - T)\wedge\alpha} \leq \abs{\int_X f (T_i - \ov{T}_i)\wedge\alpha} + \abs{\int_X f (T - \ov{T}_i)\wedge\alpha} \leq \norm{T_i - T}.
\]
showing strong convergence.

We now show the claim. It follows directly from the equality:
\[
\llbracket \max\{\vp_i + \psi, \vp + \psi_i\}\rrbracket = \llbracket \max\{ \vp_i, \vp\}\rrbracket + \llbracket \max\{\psi_i, \psi\}\rrbracket.
\]
To see this, note first that the $\geq$ inequality is clear from the definitions. The $\leq$ inequality follows from applying $\llbracket \cdot\rrbracket$ to the pointwise inequality:
\begin{align*}
\max\{\vp_i + \psi, \vp + \psi_i\} &\geq \max\{\vp_i + \psi, \vp\} + \max\{\vp_i + \psi, \psi_i\}\\
& \geq \max\{\vp_i, \vp\} + \max\{\psi, \vp\} + \max\{\vp_i, \psi_i\} + \max\{\psi, \psi_i\},
\end{align*}
and using the essential disjointedness of the pairs $[\vp_i], [\psi_i]$ and $[\vp], [\psi]$.
\end{proof}

\begin{rmk}\label{MeStupid}
As discussed immediately after Conjecture \ref{conjecture}, it seems optimistically true that strong convergence is actually equivalent to $\norm{\cdot}$-convergence; if the Hausdorff dimension of the intersections of the supports of $T_i - \ov{T_i}$ and $T - \ov{T}_i$ is $< 2n -2$, then one could separate these components using indicator functions.
\end{rmk}

Compare Proposition \ref{strong} with what is known about weak convergence of currents of integration along divisors, e.g. \cite[Chpt. 13]{Dem6}.\\

In order to show that $\delta\mathcal{T}$ is actually a Banach space, we need to verify that the metric induced by $\norm{\cdot}$ is complete. We will do this by comparing $\norm{\cdot}$ with the (pseduo-)metric of Darvas-Di Nezza-Lu \cite{DDL4} on the space of singularity types; in particular, we will show that the norm dominates the $d_\mathcal{S}$-distance, so that we can use \cite[Theorem 4.9]{DDL4}.

\begin{prop}\label{equivalent}
Suppose that $\omega$ is a sufficiently large K\"ahler form, and that $[\vp], [\psi]$ are plurisupported singularity types such that $\vp, \psi \in \PSH(X,\omega)$. Then there exists a constant $C(n)$, depending only on $n$, such that:
\begin{equation}\label{equivalent_1}
d_\mathcal{S}(\vp, \psi) \leq C(n) \norm{\vp - \psi}_\omega.
\end{equation}
In the reverse direction, if either $[\vp] \leq [\psi]$ or $[\psi] \leq [\vp]$, then:
\begin{equation}\label{equivalent_2}
d_\mathcal{S}(\vp, \psi) \geq \frac{1}{n+1}\norm{\vp - \psi}_\omega.
\end{equation}
\end{prop}
\begin{proof}
By \cite[Prop. 3.5]{DDL4}, we have:
\[
d_\mathcal{S}(\vp, \psi) \leq \sum_{j=1}^n 2\int_X\langle\omega_{\max\{\vp, \psi\}}^{j}\wedge\omega^{n-j}\rangle - \int_X \langle \omega^j_\vp\wedge\omega^{n-j}\rangle - \int_X \langle\omega_\psi^j\wedge\omega^{n-j} \rangle.
\]
Increase $\omega$ so that we can find some $\xi\in\llbracket \max\{\vp, \psi\}\rrbracket\cap\PSH(X,\omega)$. Since $\llbracket\max\{\vp, \psi\}\rrbracket\leq [\max\{\vp, \psi\}]$, by \cite{WN2} we have:
\[
d_\mathcal{S}(\vp, \psi) \leq \sum_{j=1}^n 2\int_X\langle\omega_\xi^{j}\wedge\omega^{n-j}\rangle - \int_X \langle \omega_\vp^j\wedge\omega^{n-j}\rangle - \int_X \langle\omega_\psi^j\wedge\omega^{n-j} \rangle.
\]
Let $T, S,$ and $R$ be plurisupported representatives of $[\vp],[\psi],$ and $[\xi]$, respectively, and let $\theta_T, \theta_S$, and $\theta_R$ be smooth representatives of $[T], [S]$, and $[R]$, respectively. Then for any $j$ we have:
\[
\int_X \langle \omega_\vp^j\wedge\omega^{n-j}\rangle = \int_X (\omega - \theta_T)^j\wedge\omega^{n-j} = \int_X \omega^n + \sum_{k=1}^{j} {j\choose k} (-1)^k \int_X\theta_T^k\wedge\omega^{n-k}
\]
as $T$ is plurisupported. Combining this with the similar expressions for the $\psi$ and $\xi$ terms gives:
\begin{align*}
2\int_X\langle\omega_\xi^{j}\wedge\omega^{n-j}\rangle& - \int_X \langle \omega_\vp^j\wedge\omega^{n-j}\rangle - \int_X \langle\omega_\psi^j\wedge\omega^{n-j} \rangle\\
= &\sum_{k=1}^{j} {j\choose k} (-1)^{k+1} \left(\int_X\theta_T^k\wedge\omega^{n-k} + \int_X\theta_S^k\wedge\omega^{n-k} - 2\int_X \theta_R^k\wedge\omega^{n-k}\right).
\end{align*}
 Up to increasing $\omega$ again, we can assume that the classes $[\omega \pm \theta_T], [\omega \pm \theta_S], [\omega \pm \theta_R]$ are K\"ahler. It follows from Lemma \ref{kkforms} that:
\[
\abs{\int_X (\theta_T^k - \theta_R^k)\wedge\omega^{n-k}} = \abs{\int_X (T - R)\wedge\left(\sum_{i = 0}^{k-1} \theta_T^i\wedge\theta_R^{k-i - 1}\right)\wedge\omega^{n-k}} \leq k\cdot 3^{k-1} \int_X (T-R)\wedge\omega^{n-1},
\]
as $(T-R)\wedge\omega^{n-k}$ is a (strongly) positive current. The same holds for the $S-R$ terms, so we get:
\begin{align*}
d_S(\vp, \psi) \leq \sum_{j=1}^n 2\int_X\langle\omega_\xi^{j}\wedge\omega^{n-j}\rangle - \int_X \langle \omega_\vp^j\wedge\omega^{n-j}\rangle - \int_X \langle\omega_\psi^j\wedge\omega^{n-j} \rangle \leq C(n)\norm{T - S}_\omega,
\end{align*}
proving \eqref{equivalent_1}.

Seeing \eqref{equivalent_2} is trivial. Suppose without loss of generality that $\vp \geq \psi$. Then by \cite[Lemma 3.4]{DDL4}, we have:
\begin{align*}
d_\mathcal{S}(\vp, \psi) &= \frac{1}{n+1} \sum_{j=1}^n \int_X \langle \omega_\vp^j\wedge\omega^{n-j}\rangle - \int_X \langle \omega_\psi^j\wedge\omega^{n-j}\rangle\\
& \geq \frac{1}{n+1}\left(\int_X \langle\omega_\vp\wedge\omega^{n-1}\rangle - \int_X \langle\omega_\psi\wedge\omega^{n-1}\rangle\right)\\
& = \frac{1}{n+1}\norm{\psi - \vp}_\omega.
\end{align*}
\end{proof}

\begin{lem}\label{kkforms}
Suppose that $\theta_1, \ldots, \theta_n$ are closed, smooth real $(1,1)$-forms on $X$, and that $\omega$ is a K\"ahler form such that:
\[
\left[\omega \pm \theta_i\right]
\]
is a  K\"ahler class, for each $i$. Then for any $0\leq k \leq n$ and any closed, positive $(n-k, n-k)$-current $T_k$, we have:
\[
\abs{\int_X \theta_K \wedge T_k} \leq 3^k \int_X \omega^k \wedge T_k,
\]
where here $\theta_K := \theta_1\wedge\ldots\wedge\theta_k$. 
\end{lem}
\begin{proof}
We proceed by induction  on $k$. The case $k = 1$ is clear by Stokes' theorem:
\[
\int_X (\omega \pm \theta_i)\wedge T_1 \geq 0.
\]
Suppose then the conclusion is true for $0 \leq k < n$; then for $k+1$, we compute:
\begin{align*}
\abs{\int_X \theta_{K+1} \wedge T_{k+1}} &\leq \abs{\int_X \theta_K\wedge (\omega + \theta_{k+1})\wedge T_{k+1}} + \abs{\int_X \theta_K \wedge \omega\wedge T_{k+1}}\\
&\leq  3^k \int_X (\omega + \theta_{k+1})\wedge \omega^k \wedge T_{k+1}+ 3^k\int_X \omega^{k+1}\wedge T_{k+1}\\
& \leq 3^{k+1}\int_X \omega^{k+1}\wedge T_{k+1},
\end{align*}
using the induction hypothesis 3 times (once with the closed positive current $\ti{\omega}\wedge T_{k+1}$, where $\ti{\omega}\in [\omega + \theta_{k+1}]$ is K\"ahler, which we can do using Stokes' Theorem).
\end{proof}

The inequality \eqref{equivalent_2} seems likely to be false in general, as $\llbracket\max\{\vp, \psi\}\rrbracket < [\max\{\vp, \psi\}]$. Thus, it is it unlikely that $d_\mathcal{S}$ and $\norm{\cdot}$ are truly equivalent; however, it is easy to see that they are close enough that they induce the same topologies.

Before proving this, we need a brief proposition to address the fact that $d_\mathcal{S}$ is really only a quasi-metric on the space of singularity types.

\begin{prop}\label{ignore}
Suppose that $\vp \in\PSH(X,\theta)$ and that $\vp \leq 0$. Define the singularity type envelope:
\[
P_\theta[\vp] := \sup\{v\in \PSH(X,\theta)\ |\ v\leq 0, v\leq \vp + \mathcal{O}(1)\}^*,
\]
and the model type envelople
\[
\mathcal{C}_\theta(\vp) := \lim_{\e\rightarrow 0^+} P_\theta[(1-\e)\vp + \e V_\theta].
\]
Then $\llbracket\vp\rrbracket = \llbracket P_\theta[\vp]\rrbracket = \llbracket \mathcal{C}_\theta(\vp)\rrbracket$.
\end{prop}
\begin{proof}
Note that, by Theorem \ref{transfer}:
\[
P_\theta[\vp] = \left(\lim_{C\rightarrow\infty} P_\theta(\min\{0, \vp + C\})\right)^* = \psi  + \left(\lim_{C\rightarrow\infty} P_{\theta - \sigma}(\min\{-\psi, \vp - \psi + C\})\right)^*,
\]
where $\sigma_{\psi}$ is a plurisupported representative of $\llbracket \vp \rrbracket$. It is easy to see that there is some $C_0$ such that:
\[
\vp - \psi - C_0 \leq \left(\lim_{C\rightarrow\infty} P_{\theta - \sigma}(\min\{-\psi, \vp - \psi + C\})\right)^* =: \eta,
\]
so that $\llbracket \eta \rrbracket = 0$ and hence $\llbracket P_\theta[\vp]\rrbracket = [\psi] = \llbracket \vp\rrbracket$.

For $\mathcal{C}_\theta(\vp)$, we have from the first case that:
\[
\mathcal{C}_\theta(\vp) \leq (1-\delta)\psi + P_{\theta}(-\psi) \leq (1-\delta)\psi + C,
\]
for any $0 < \delta < 1$. Letting $\delta \rightarrow 0$ thus finishes, as it is obvious that $P_\theta[\vp] \leq \mathcal{C}_\theta(\vp)$.

\end{proof}

As a remark, if $\int_X \langle\theta_\vp^n\rangle > 0$, then $P_\theta[\vp] = \mathcal{C}_\theta[\vp]$, by \cite[Prop. 2.6]{DDL4} -- conjecturally, equality should always hold \cite[Conj. 2.5]{DDL4}.

\begin{thm}\label{converge}
Suppose that $\omega$ is a K\"ahler form, and let $\mathcal{S}_\delta$, $\delta > 0$, be as in \cite{DDL4}. Suppose that $[\vp_i], [\vp]\in \mathcal{S}_\delta$ are such that $[\vp_i]$ $d_S$-converges to $[\vp]$. Then the sequence $\llbracket \vp_i\rrbracket$ $\norm{\cdot}_\omega$-converges to $\llbracket \vp\rrbracket$.
\end{thm}
\begin{proof}
By Proposition \ref{ignore}, the singularity types $\llbracket \vp_i\rrbracket$ are independent of the choice of $\vp_i$. Choose an arbitrary subsequence of the $[\vp_i]$, which still $d_\mathcal{S}$-converges to $[\vp]$. By \cite[Thrm. 5.6]{DDL4}, there exists a further subsequence (which we will still call $[\vp_i]$) such that these $\vp_i$ can be sandwiched between two sequences $w_i \leq \vp_i \leq v_i$, $w_i, v_i\in\PSH(X,\omega)$, with the $w_i$ increasing, the $v_i$ decreasing, and both $d_\mathcal{S}$-converging to $\vp$. By \cite[Lemma 3.7]{DDL4}, we have:
\[
\abs{\int_X \langle \omega_{w_i}\wedge\omega^{n-1}\rangle - \int_X \langle\omega_\vp\wedge\omega^{n-1}\rangle}  = 
\abs{\int_X T_i\wedge\omega^{n-1} - \int_X T\wedge\omega^{n-1}} \rightarrow 0,
\]
where here $T_i$, $T$ are plurisupported representatives of $\llbracket w_i\rrbracket, \llbracket \vp\rrbracket$, respectively. It follows from Corollary \ref{ordering} that $T \leq T_i$, so that:
\[
\norm{\llbracket w_i\rrbracket  - \llbracket \vp\rrbracket }_\omega = \int_X (T_i - T)\wedge\omega^{n-1},
\]
and so $\llbracket w_i\rrbracket$ $\norm{\cdot}$-converges to $\llbracket \vp\rrbracket$; similar reasoning holds for $\llbracket v_i\rrbracket$. Then for any $\e > 0$ and $i$ sufficiently large we have:
\[
\norm{\llbracket \vp_i\rrbracket - \llbracket \vp\rrbracket} \leq  \norm{\llbracket w_i\rrbracket - \llbracket \vp_i\rrbracket} + \e \leq \norm{\llbracket w_i\rrbracket - \llbracket v_i\rrbracket} + \e \leq 3\e.
\]
Since we started by taking an arbitrary subsequence, this finishes the proof.
\end{proof}

\begin{thm}\label{Banach}
The space $\delta\mathcal{T}$ is a Banach space.
\end{thm}
\begin{proof}
Suppose that $[\vp_i - \psi_i]$ is a $\norm{\cdot}$-Cauchy sequence, with $[\vp_i], [\psi_i]$ essentially disjoint for each $i$. Let $T_i, S_i$ be plurisupported representatives of $[\vp_i], [\psi_i]$, respectively. In order to apply Proposition \ref{equivalent} and \cite[Theorem 4.9]{DDL4} we need to produce a K\"ahler form $\omega$ such that the classes $[\omega \pm T_i], [\omega\pm S_i]$ each dominate some fixed, small multiple of $[\omega]$.

As mentioned in the proof of Proposition \ref{strong}, if $[\vp_i], [\psi_i]$ and $[\vp_j], [\psi_j]$ are essentially disjoint pairs, then:
\[
\norm{(\vp_i - \psi_i) - (\vp_j - \psi_j)} = \norm{\vp_i - \vp_j} + \norm{\psi_i - \psi_j},
\]
so that the sequences $[\vp_i]$ and $[\psi_i]$ are both $\norm{\cdot}$-Cauchy; in particular, both are $\norm{\cdot}$-bounded, so it follows that the cohomology classes $[T_i], [S_i]$ are all contained in the compact set $\{[\alpha]\in \mathcal{E}(X)\ |\ \int_X \alpha\wedge\omega^{n-1} \leq C\} \subseteq H^{1,1}(X, \R)$, where here $\mathcal{E}(X)$ is the pseudoeffective cone; the same reasoning applies to $[-T_i], [-S_i]$. Hence, since the K\"ahler cone is open, there exists a K\"ahler class $[\omega]$ such that each of the classes $[\omega \pm T_i], [\omega \pm S_i]$ is K\"ahler; increasing it slightly if necessary produces the desired $\omega$.

Hence, we conclude there exists $\vp, \psi\in\PSH(X,\omega)$ such that $[\vp_i]$ and $[\psi_i]$ $d_\mathcal{S}$-converge to $[\vp]$ and $[\psi]$, respectively; by Theorem \ref{converge}, we thus have that $[\vp_i]$ $\norm{\cdot}$-converges to $\llbracket \vp\rrbracket$. By Proposition \ref{equivalent}, $[\vp_i]$ $d_\mathcal{S}$-converges to $\llbracket\vp\rrbracket$ though (after possibly increasing $\omega$), and so we conclude $[\vp] = \llbracket \vp\rrbracket$ -- this also hold for $\psi$, finishing the proof.
\end{proof}

We can summarize the contents of the previous two theorems as follows:

\begin{cor}\label{carebear}
If $[\vp_i]$, $[\vp]$ are plurisupported singularity types, then the sequence $[\vp_i]$ $\norm{\cdot}$-converges to $[\vp]$ if and only if there exists some sufficiently large K\"ahler form $\omega$ such that $\vp_i, \vp\in \PSH(X,\omega)$ and $[\vp_i]$ $d_\mathcal{S}$-converges to $[\vp]$.
\end{cor}

We conclude this section with a sufficient condition for the singular and the regular Siu decompositions to agree, to help place the pluripolar Siu decomposition in some more context. As any practitioner can tell you, most explicit currents are $\mathcal{I}$-model, and almost all geometrically interesting ones seem to be as well -- for instance, $V_\theta$ is always $\mathcal{I}$-model. As such, the pluripolar Siu decomposition should really only need to appear if one ``seeks it out," so to say.

\begin{prop}\label{I-sings}
Suppose that $\vp\in\PSH(X,\theta)$ is $\mathcal{I}$-model. By \cite{DX}, it is in the $d_\mathcal{S}$-closure of the space of analytic singularity types. Then $\llbracket \vp\rrbracket$ is a (possibly infinite) sum of currents of integration along divisors.
\end{prop}
\begin{proof}
By Corollary \ref{carebear} and Proposition \ref{strong}, $\llbracket \vp\rrbracket$ is a strong limit of a sequence of currents of integration along divisors; this can only converge to a sum of divisors.
\end{proof}


\section{Estimate on the Pluripolar Mass of Envelopes}


In this section, we generalize the second part of the proof of \cite{WN1}, establishing an inequality for the pluripolar mass of certain envelopes of psh functions (Theorem \ref{ppbound_intro}/ \ref{ppbound}). We then use this to show the McKinnon-Roth inequality \cite{MR} for arbitrary pseudoeffective classes on compact K\"ahler manifolds, and prove directly Demailly's weak transcendental holomorphic Morse inequality when the class $[\beta]$ is only modified nef and plurisupported (c.f. \cite{BFJ} or \cite[Prop. 1.1]{Xiao3}). Lastly, we show that this is enough to repeat the arguments in the appendix of \cite{WN1}, proving the BDPP conjecture under the assumption that $X$ admits a plurisupported K\"ahler class.

\subsection{Pluripolar Mass Inequality}

We begin with the mass inequality, which gives an upper-bound on the pluripolar mass of envelopes when the cohomology class is shifted. Our proof is a generalization of \cite{WN1}, using the recent technical improvements of \cite{DDL2} (see also \cite{DT}, where they use a similar technique in the proof of their main result). 

A weaker version of Theorem \ref{ppbound} can be found in the author's thesis \cite[Theorem 3.2.1]{McC3}; compared with that version, the proof has been simplified (in such a way that it no longer relies on the results in \cite{DT}), and the conclusion has been strengthened, now applying to pseudoeffective classes (\cite[Theorem 3.2.1]{McC3} was only stated for K\"ahler classes).

We begin with a simple lemma, which is essentially just \cite[Lemma 4.4]{DDL4}:
\begin{lem}\label{lemma}
Let $X^n$ be a compact K\"ahler manifold. Suppose that $\theta, \sigma, \rho$ are each smooth, closed real $(1,1)$-forms such that $[\theta], [\sigma],$ and $[\rho]$ are pseudoeffective classes, and additionally:
\[
\theta = \sigma+ \rho.
\]
Then if $\psi\in\PSH(X,\theta), \eta\in\PSH(X,\rho)$ are such that $P_\sigma(\psi-\eta)\not=-\infty$, we have:
\[
\langle \sigma_{P_{\sigma}(\psi-\eta)}^n\rangle \leq \langle\theta_\psi^n\rangle.
\]
\end{lem}
\begin{proof}
Set $\vp := P_\sigma(\psi-\eta)$. By the proof of \cite[Lemma 4.4]{DDL4} (see also \cite[Prop 3.10]{LN}), the measure $\langle \sigma_\vp^n\rangle$ is supported on the contact set $\{\vp +\eta = \psi\}$.

Now, note that $\vp + \eta\in\PSH(X,\theta)$ and that $\vp + \eta \leq \psi$. So:
\[
\langle\sigma_\vp^n\rangle \leq \chi_{\{\vp + \eta = \psi\}}\langle(\sigma_\vp + \rho_\eta)^n\rangle = \chi_{\{\vp + \eta = \psi\}}\langle\theta_{\vp + \eta}^n\rangle = \chi_{\{\vp + \eta = \psi\}}\langle\theta_{\psi}^n\rangle \leq \langle\theta_\psi^n\rangle,
\]
by pluri-locality of the non-pluripolar product \cite{BEGZ}.
\end{proof}

\begin{thm}\label{ppbound}
Let $X^n$ be a compact K\"ahler manifold, and suppose that $[\theta], [\sigma], [\rho]$ are pseudoeffective classes on $X$ such that $\theta = \sigma+\rho$. Let $\xi, \psi\in\PSH(X,\theta)$ and $\eta\in\PSH(X,\rho)$ be such that $P_\sigma(\psi - \eta) \in\PSH(X,\sigma)$ is not $-\infty$, $\psi\leq\xi$ (so that $P_\sigma(\xi - \eta) \not=-\infty$ also), and $\eta$ has small unbounded locus. Then we have the following estimate:
\begin{equation}\label{blah}
\int_X \langle\sigma^n_{P_\sigma(\xi - \eta)}\rangle - \int_X \langle \sigma_{P_\sigma(\psi - \eta)}^n\rangle \leq \int_X\langle\theta^n_\xi\rangle - \int_X \langle\theta_\psi^n\rangle.
\end{equation}
\end{thm}
\begin{proof}
If $[\sigma]$ is only pseudoeffective, the left-hand side of \eqref{blah} is 0, and so we are done by \cite[Thrm. 1.2]{WN2}. We can thus assume that $[\sigma]$ is big. Without loss of generality, we may also assume that $\psi, \xi, \eta \leq 0$.

Let $\vp := P_\sigma(\psi - \eta)$ and choose a decreasing sequence of smooth functions $\ti{\psi}_j\in C^\infty(X)$ such that $\ti{\psi}_j \searrow \psi$. Define:
\[
\psi_j := P_\theta(\min\{\ti{\psi}_j, \xi\}).
\]
Then each $\psi_j\in\PSH(X,\theta)$ has the same singularity type as $\xi$ and we still have $\psi_j\searrow\psi$ as $j\rightarrow\infty$. Similarly, define:
\[
\vp_j := P_\sigma(\psi_j - \eta)
\]
and note that $[\vp_j] = [P_\sigma(\xi - \eta)]$ and $\vp_j\searrow \vp$.

Let:
\[
V_\sigma := \sup\{v\in\PSH(X,\sigma)\ |\ v\leq 0\}^* := P_\sigma(0)
\]
be the $\sigma$-psh function with minimal singularities; since $[\sigma]$ is big, by Demailly's approximation theorem \cite{Dem3}, $V_\sigma$ has small unbounded locus. We can find then an increasing sequence of open sets $U_i$, $i\rightarrow\infty$, such that $X\setminus\left(\cup_i U_i\right)$ is contained inside a closed, pluripolar set, and such that $V_\sigma \geq V_\sigma + \eta \geq -C_i$ on $U_i$, for some sequence $0 \leq C_i \rightarrow \infty$.

We will now recall the cut-off functions of \cite{DDL2}. Given any $R, \e > 0$ and $v\in\PSH(X,\sigma)$, we write:
\[
h_{R,\e}(v) := \frac{\max\{v - V_\sigma + R, 0\}}{\max\{v - V_\sigma + R, 0\} + \e}.
\]
Note that $0\leq h_{R,\e}(v) \leq \chi_{\{v > V_\sigma -R\}}$, so that $v = \max\{v, V_\sigma - R\}$ on the (pluri-open) set $\{h_{R,\e}(v) > 0\} = \{v > V_\sigma - R\}$. Further observe that $h_{R,\e}(v)$ is quasi-continuous, and increases as $\e\rightarrow0$ to $\chi_{\{v > V_\sigma -R\}}$. Finally, and crucially, note that if a sequence $v_j\in\PSH(X,\sigma)$ converges in capacity to $v$, then the sequence $h_{R,\e}(v_j)$ also converge in capacity to $h_{R,\e}(v)$.

Fix now some index $i$. Then we have:
\[
\int_X \langle \sigma_\vp^n\rangle \geq \int_{U_i} h_{R,\e}(\vp)\langle \sigma_\vp^n\rangle = \int_{U_i} h_{R,\e}(\vp) \sigma_{\max\{\vp, V_\sigma - R\}}^n,
\]
by locality of the non-pluripolar product \cite{BEGZ}. By Xing's convergence lemma \cite{Xing} (c.f. \cite[Thrm. 4.26]{GZ_book}), the right-hand side is then equal to:
\begin{equation}\label{bbb}
\int_{U_i}h_{R,\e}(\vp) \sigma_{\max\{\vp, V_\sigma - R\}}^n = \lim_{j\rightarrow\infty}\int_{U_i}h_{R,\e}(\vp_j)  \sigma_{\max\{\vp_j, V_\sigma - R\}}^n = \lim_{j\rightarrow\infty}\int_{U_i} h_{R,\e}(\vp_j)  \langle\sigma_{\vp_{j}}^n\rangle
\end{equation}
\[
 = \lim_{j\rightarrow\infty} \left(\int_X \langle\sigma_{\vp_{j}}^n\rangle - \int_X (1 - \chi_{U_i} h_{R,\e}(\vp_j))  \langle\sigma_{\vp_{j}}^n\rangle\right)
\]
as $\max\{\vp_j, V_\sigma -  R\} \geq V_\sigma-R \geq -C_i-R$ on $U_i$. 

By Lemma \ref{lemma} we have that:
\[
\langle \sigma_{\vp_j}^n\rangle \leq \langle \theta_{\psi_j}^n\rangle,
\]
and so:
\[
\int_X \langle \sigma_{\vp}^n\rangle  \geq \lim_{j\rightarrow\infty} \left(\int_X \langle\sigma_{\vp_{j}}^n\rangle - \int_X (1 - \chi_{U_i} h_{R,\e}(\vp_j))  \langle\theta_{\psi_{j}}^n\rangle\right).
\]
Recalling that $[\psi_j] = [\xi]$ and $[\vp_j] = [P_\sigma(\xi - \eta)]$, we can apply \cite{WN2} to get:
\[
\int_X \langle \sigma_{\vp}^n\rangle \geq \int_X \langle \sigma_{P_\sigma(\xi-\eta)}^n\rangle - \int_X\langle \theta_\xi^n\rangle + \lim_{j\rightarrow\infty} \int_{U_i} h_{R,\e}(\vp_j)  \langle \theta_{\psi_{j}}^n\rangle.
\]
Now:
\[
\{h_{R,\e}(\vp_j) > 0\} = \{\vp_j > V_\sigma - R\} \subseteq \{\psi_j > V_\sigma + \eta - R\},
\]
so by pluri-locality we again get:
\[
\int_{U_i} h_{R,\e}(\vp_j)  \langle \theta_{\psi_{j}}^n\rangle = \int_{U_i} h_{R,\e}(\vp_j)  \theta_{\max\{\psi_{j}, V_\sigma + \eta - R\}}^n.
\]
On $U_i$ we also have 
that $V_\sigma + \eta - R \geq -C_i - R$, so we can apply Xing's lemma to take the limit again, and, after using pluri-locality once more, get:
\[
\int_X \langle \sigma_{\vp}^n\rangle \geq \int_X \langle \sigma_{P_\sigma(\xi-\eta)}^n\rangle - \int_X\langle \theta_\xi^n\rangle + \int_{U_i} h_{R,\e}(\vp) \langle \theta_{\psi}^n\rangle.
\]
Letting $\e\rightarrow0$ and then $R\rightarrow\infty$, we see that $h_{R,\e}(\vp)\nearrow \chi_{\{\vp > \infty\}}$. Since the non-pluripolar product doesn't charge the pluripolar set $\{\vp = -\infty\}$, we thus have:
\[
\int_X \langle \sigma_{\vp}^n\rangle \geq \int_X \langle \sigma_{P_\sigma(\xi-\eta)}^n\rangle - \int_X\langle \theta_\xi^n\rangle + \int_{U_i}\langle \theta_{\psi}^n\rangle.
\]
Letting $i\rightarrow\infty$ and applying the same reasoning finishes the proof.
\end{proof}

\begin{exmp}\label{ASDFASDF}
We provide an example to show that the inclusion of the function $\eta$ in Theorem \ref{ppbound} is necessary when $[\rho]$ is not nef. Take $X$ to be the blow-up of $\CP^n$, $n > 1$, at a single point;
\[
\mu: X\rightarrow \CP^n.
\]
Let $E\subset X$ be the exceptional divisor and $\omega_{FS}$ the Fubini-Study metric on $\CP^n$. By \cite{DP}, there exists a smooth form $R\in[E]$ such that $\mu^*\omega_{FS} - \e R$ is a K\"ahler form for all $\e > 0$ sufficiently small -- fix a number $t > 0$ such that $\mu^*\omega_{FS} - 2tR$ and $\mu^*\omega_{FS} - tR$ are both K\"ahler forms. Let $\xi$ be the unique (up to normalization) $R$-psh function:
\[
R + i\p\pbar\xi = \llbracket E\rrbracket.
\]

We now define $\sigma = \mu^*\omega_{FS}$ and $\theta = \mu^*\omega_{FS} + tR$, so that $[\theta - \sigma] = t[E]$ is pseudoeffective. It follows that $t\xi$ and $2t\xi$ are both $\theta$-psh and $\sigma$-psh. We have:
\[
\int_X \langle \sigma_{t\xi}^n\rangle - \int_X \langle \sigma_{2t\xi}^n\rangle = \int_X (\mu^*\omega_{FS} - tR)^n - \int_X (\mu^*\omega_{FS} - 2tR)^n.
\]
By a simple intersection number calculation, we have that:
\[
\int_X(\mu^*\omega_{FS} - \e R)^n = \int_{\CP^n}\omega_{FS}^n - \e^n = 1- \e^n,
\]
for all $\e > 0$ sufficiently small, so that:
\[
\int_X \langle \sigma_{t\xi}^n\rangle - \int_X \langle \sigma_{2t\xi}^n\rangle = (1-t^n) - (1- (2t)^n) = (2^n - 1)t^n.
\]
Similarly, we compute that:
\[
\int_X \langle \theta_{t\xi}^n\rangle - \int_X \langle \theta_{2t\xi}^n\rangle = 1 - (1-t^n) = t^n.
\]
As $n > 1$, we see that inequality \eqref{blah} does not hold in this case.

Adding even a large multiple of $\omega_{FS}$ to $\theta$ in the above example will not change the value of the difference; as such, the inequality still fails without $\eta$, even if $[\rho]$ is big with arbitrarily large volume (but not nef).
\end{exmp}


\subsection{The McKinnon-Roth Inequality}


By slightly updating the argument in \cite{McC3}, we can use Theorem \ref{ppbound} to prove the McKinnon-Roth inequality \cite{MR} for any pseudoeffective class $[\theta]$. As with Theorem \ref{ppbound}, this result was shown for K\"ahler classes in \cite{McC3}. 

\begin{thm}
Suppose that $[\theta]$ is a pseudoeffective class on $X$, and that:
\[
\mu: \ti{X} \rightarrow X
\]
is the blow-up of $X$ at a point $x\in X$; let $E$ be the exceptional divisor. Then:
\begin{equation}\label{MR}
\vol_{\ti{X}}(\mu^*\theta - t[E])\geq \vol_X(\theta) - t^n\ \ \forall t \in [0, \zeta_x(\theta)],
\end{equation}
where here $\zeta_x(\theta)$ is the Nakayama constant (also called the pseudoeffective threshold).
\end{thm}
\begin{proof}
Fix some K\"ahler form $\omega$ such that $\omega - \theta$ is K\"ahler. By explicit construction, for any $c > 0$ sufficiently small there exists a function $\psi_x\in\PSH(X,\omega)$, smooth on $X\setminus\{x\}$, such that:
\[
\psi_x(z) = c\log \abs{z} + h(z)
\]
in local coordinates centered at $x$, where here $h(z)$ is some smooth, bounded function. Up to scaling $\omega$, we may take $c = 1$.

By the Poincare-Lelong formula, we have that:
\begin{equation}\label{ssss}
i\p\pbar\mu^*\psi_x = \llbracket E \rrbracket - \eta,
\end{equation}
for some smooth form $\eta \in [E]$. By Theorem \ref{transfer} (really just Lemma \ref{hoponpop}), we have:
\[
P_{\mu^*\theta}(t\mu^*\psi_x) = t\mu^*\psi_x + V_{\mu^*\theta - t\eta}.
\]
Thus, the left-hand side of \eqref{MR} can be rewritten as:
\begin{align*}
\vol_{\ti{X}}(\mu^*\theta - t[E]) &= \int_{\ti{X}} \langle (\mu^*\theta - t\eta + i\p\pbar V_{\mu^*\theta - t\eta})^n\rangle = \int_{\ti{X}} \langle (\mu^*\theta + i\p\pbar P_{\mu^*\theta}(t\mu^*\psi_x))^n\rangle\\
&=\int_X \langle (\theta + i\p\pbar P_\theta(t\psi_x))^n\rangle,
\end{align*}
using standard properties of modifications. Our desired inequality \eqref{MR} can now be restated as:
\begin{equation}\label{umm}
\vol(\theta) - \int_X \langle (\theta + i\p\pbar P_\theta(t\psi_x))^n\rangle \leq t^n.
\end{equation}

We break the proof into two cases -- $t\leq 1$ and $t > 1$. In the first case, we have that $t\psi_x\in\PSH(X,\omega)$ and so we can compute:
\[
t^n = \int_X \omega^n - \int_X \langle \omega_{t\psi_x}^n\rangle,
\]
by \cite{Dem4}. Our desired inequality \eqref{umm} can thus be rewritten again as:
\[
\vol(\theta) - \int_X \langle (\theta + i\p\pbar P_\theta(t\psi_x))^n\rangle \leq \vol(\omega) - \int_X \langle \omega_{t\psi_x}^n\rangle;
\]
which is just \eqref{blah} with $\xi = \eta = 0$. Similarly, if $t > 1$, then $t\psi_x \in \PSH(X,t\omega)$ and we still have:
\[
t^n = \int_X (t\omega)^n - \int_X \langle (t\omega_{t\psi_x})^n\rangle.
\]
Hence \eqref{umm} follows again from Theorem \ref{ppbound}. 
\end{proof}


\subsection{The Argument of Witt-Nystr\"om}


We can now combine Lemma \ref{subtraction} and Theorem \ref{ppbound} to prove Theorem \ref{WN_MN}.

\begin{thm}\label{WN_MN}
Suppose that $[\alpha]$ and $[\beta]$ are big classes, and that $[\beta]$ also admits a plurisupported current. Then if $[\alpha - \beta]$ is pseudoeffective, we have:
\[
\vol(\alpha - P(\beta)) \geq \vol(\alpha) - n\langle\alpha^{n-1}\rangle\wedge P(\beta),
\]
where $\langle \alpha^{n-1}\rangle = \langle \alpha^{n-1}_{V_\alpha}\rangle$ is the moving product of \cite{BDPP}.
\end{thm}
\begin{proof}
Let $\psi\in\PSH(X,\beta)$ be the potential of a plurisupported current, and let $\rho_\eta = \llbracket N(\beta)\rrbracket$, the negative part of the divisorial Zariski decomposition. It is easy to see that $(\beta - \rho)_{\xi - \eta}$ is a plurisupported current in $P(\beta)$. Theorem \ref{transfer} now gives us that:
\[
P_\alpha(V_\alpha + \psi - V_\beta) = \psi - \eta + P_{\alpha - P(\beta)}(V_\alpha - V_{P(\beta)})
\]
as $V_\beta = V_{P(\beta)} + \eta$, again by Theorem \ref{transfer} (or really just Lemma \ref{hoponpop} for this one). 

Let $V_{P(\beta)} := V_{\beta - \rho}$ and $V_{\alpha - P(\beta)} := V_{\alpha - \beta + \rho}$. By definition, we have $V_{\alpha - P(\beta)} + V_{P(\beta)} \leq V_\alpha$, so:
\[
V_{\alpha - P(\beta)} \leq P_{\alpha - P(\beta)}(V_\alpha - V_{P(\beta)}),
\]
which implies that in fact:
\[
[V_{\alpha - P(\beta)}] = [P_{\alpha - P(\beta)}(V_\alpha - V_{P(\beta)})]
\]
as $V_{\alpha - P(\beta)}$ has minimal singularity type. Thus, $[P_\alpha(V_\alpha + \psi - V_\beta)] = [\psi - \eta + V_{\alpha - P(\beta)}]$.

Now by Theorem \ref{ppbound} with $\theta = \alpha + \beta$, $\sigma = \alpha$, $\rho = \beta$, and $\xi = V_\alpha + V_\beta$, $\psi\ \text{(in Theorem \ref{ppbound})} = V_\alpha + \psi\ \text{(from the start of the proof)}$, $\eta = V_\beta$ we have:
\[
\vol(\alpha) - \int_X \langle \alpha^n_{P_\alpha(V_\alpha + \psi - V_\beta)}\rangle \leq \int_X \langle (\alpha + \beta)_{V_\alpha + V_\beta}^n\rangle - \vol(\alpha),
\]
as $\langle (\alpha_{V_\alpha} + \beta_\psi)^n\rangle = \langle \alpha_{V_\alpha}^n\rangle = \vol(\alpha)$ (by \cite{BEGZ}), and $P_\alpha(V_\alpha) = V_\alpha$. Rearranging this using multilinearity of the non-pluripolar product gives us:
\[
\int_X \langle \alpha^n\rangle - \sum_{k=1}^n {n \choose k} \int_X \langle \alpha^{n-k} \wedge P(\beta)^k \rangle \leq \int_X \langle \alpha^n_{P_\alpha(V_\alpha + \psi - V_\beta)}\rangle,
\]
as $\langle \alpha^{n-k}\wedge P(\beta)^k\rangle = \langle \alpha_{V_\alpha}^{n-k}\wedge \beta^{k}_{V_\beta}\rangle$ by \cite{BEGZ}. By the first paragraph, $[P_\alpha(V_\alpha + \psi - V_\beta)] = [\psi - \eta + V_{\alpha - P(\beta)}]$, so by \cite{BEGZ} we have:
\begin{equation}\label{okie}
\int_X \langle \alpha^n\rangle - \sum_{k=1}^n {n \choose k} \int_X \langle \alpha^{n-k} \wedge P(\beta)^k \rangle \leq \int_X \langle \alpha^n_{\psi - \eta + V_{\alpha - P(\beta)}}\rangle = \int_X \vol(\alpha - P(\beta)).
\end{equation}
In order to conclude, replace $\beta$ by $t\beta$ -- while the moving product is not additive, it is multiplicative (as $V_{t\beta} = tV_\beta$ for $0 < t \leq 1$), and so is the divisorial Zariski decomposition, so \eqref{okie} becomes:
\[
\int_X \langle \alpha^n\rangle - \sum_{k=1}^n t^k{n \choose k} \int_X \langle \alpha^{n-k} \wedge P(\beta)^k \rangle \leq \int_X \vol(\alpha - tP(\beta)).
\]
We can now conclude by using a basic calculus lemma, which can be found in the first version of \cite{WN1} on the arXiv.
\end{proof}


\subsection{The Conjecture of Boucksom--Demailly--Pa\u un--Peternell}


We finally show that a slight modification of the proof of Theorem A.2 in the appendix of \cite{WN1} can be used to show that the BDPP conjecture holds if $X$ admits a plurisupported K\"ahler class.

\begin{thm}\label{crank}
 Let $Y$ be a compact K\"ahler manifold, and suppose that $\mu: X\rightarrow Y$ is a proper modification such that $X$ admits a plurisupported current $T$ with $[T]$ a K\"ahler class. Then the BDPP conjecture holds on $Y$.
\end{thm}
\begin{proof}
Suppose that $[\alpha]\in H^{1,1}(Y, \R)$ is big. Consider a sequence of approximate Zariski decompositions of $[\mu^*\alpha]$:
\[
\pi_k: X_k \rightarrow X,\ \ \ [\mu_k^*\alpha] = [\alpha_k] + [ E_{k}],
\]
where here $\mu_k := \mu\circ \pi_k$, the $\alpha_k$ are K\"ahler forms, and the $E_{k}$ are effective, exceptional $\Q$-divisors. By \cite[Thrm. A.3]{WN1}, it will be enough to show that:
\[
\lim_{k\rightarrow\infty} \int_{X_k}\alpha_k^{n-1}\wedge\llbracket E_k\rrbracket = \langle \alpha^{n-1}\rangle \cdot \alpha - \vol(\alpha)
\]
is equal to 0.

Scale $T$ as necessary so that the classes $[T \pm \mu^*\alpha]$ are K\"ahler, and choose a smooth K\"ahler form $\omega\in[T]$. Then the line segment $[\alpha_k + t(\pi^*_k T + \llbracket E_k\rrbracket)]$, $t\in [0, 1]$, has endpoints in the nef cone, and so by convexity is entirely contained inside the nef cone. Applying Theorem \ref{WN_MN} with $[\beta] = [t \pi_*T]$ gives: 
\begin{align*}
\vol(\alpha) &\geq \vol(\alpha_k + t\llbracket E_k\rrbracket)\\
&\geq \int_{X_k} (\alpha_k + t(\pi_k^*\omega + R_k))^n - nt\int_{X_k} (\alpha_k + t(\pi^*_k \omega + R_k))^{n-1}\wedge \pi_k^*\omega \tag{\theequation}\label{piquwrteqwerrt}\end{align*}
where here $R_k$ is a smooth form in $[E_k]$.

We wish to expand the integrals on the right-hand side of \eqref{piquwrteqwerrt}. Let $0 < \e < 1$. Note that:
\[
[\pi^*_k \omega + 2\alpha_k - R_k] = [\pi^*_k(\omega - \mu^*\alpha) + 3\alpha_k] \text{ and } [\pi^*_k \omega + 2\alpha_k + R_k] = [\pi^*_k (\omega + \mu^*\alpha) + \alpha_k],
\]
are both K\"ahler classes, and so by Lemma \ref{kkforms}, we have:
\begin{align*}
\abs{\int_{X_k} \alpha_k^i\wedge (\pi_k^*\omega + \e\alpha_k)^j \wedge R_k^\ell} &\leq C(n)\int_{X_k} \alpha_k^i\wedge (\pi_k^* \omega + \e\alpha_k)^{j} \wedge(2\alpha_k + \pi^*_k \omega)^\ell\\
&\leq C(n, i, j, \ell) \sum_{r=0}^n \int_{X_k} \alpha_k^r\wedge \pi^*_k\omega^{n-r} 
\end{align*}
Both $\alpha_k$ and $\pi^*_k\omega$ are semi-positive, and so since $[\pi^*_k T - \alpha_k]$ is pseudoeffective we have:
\[
\int_{X_k} \alpha_k^n \leq \int_{X_k} \pi^*_k\omega\wedge\alpha_k^{n-1} \leq \ldots \leq \int_{X_k} \pi_k^*\omega^n = \int_X \omega^n.
\]
Letting $\e\rightarrow 0$ now shows that each of the terms in the expansion of \eqref{piquwrteqwerrt} is bounded independent of $k$:
\[
\abs{\int_{X_k} \alpha_k^i\wedge \pi_k^*\omega^j \wedge R_k^\ell} \leq C
\]

We can now expand the terms on the right-hand side of \eqref{piquwrteqwerrt}:
\begin{align*}
\vol(\alpha) - \vol(\alpha_k) &\geq nt\int_{X_k} (\pi_k^*\omega + R_k)\wedge\alpha_k^{n-1} - nt\int_{X_k} \alpha_k^{n-1}\wedge \pi_k^*\omega - C t^2\\
& = nt \int_{X_k} \alpha_k^{n-1}\wedge \llbracket E_k\rrbracket - Ct^2
\end{align*}
Up to possibly increasing $C$, we can assume that $\int_{X_k} \alpha_k^{n-1}\wedge\llbracket E_k\rrbracket \leq \frac{2C}{n}$, and so setting $t := \frac{n}{2C}\int_{X_k} \alpha_k^{n-1}\wedge\llbracket E_k\rrbracket$ gives:
\[
\left(\int_{X_k} \alpha_k^{n-1}\wedge\llbracket E_k\rrbracket\right)^2 \leq \frac{4C}{n^2} \left(\vol(\alpha) - \int_{X_k} \alpha_k^n\right),
\]
which tends to 0 as $k\rightarrow\infty$, since $\alpha_k$ is an approximate Zariski decomposition for $\mu^*\alpha$ and $\vol(\mu^*\alpha) = \vol(\alpha)$.

\end{proof}


\section{Existence of Plurisupported Currents}


In this section, we attempt to study the existence of plurisupported currents by studying the cone of plurisupported currents in $\mathcal{E}\in H^{1,1}(X,\R)$, the pseduoeffective cone. We show that any big class can be naturally decomposed (in a non-unique way) into a sum of a plurisupported class and a ``deficient" class in Corollary \ref{decomp}. We deduce that such deficient classes present a type of obstruction to the existence of plurisupported classes, but this obstruction is seen to be insufficient to completely determine their existence. We conclude by discussing a starting point for a possible alternate approach.


\subsection{The Singular Decomposition}


We denote by $\mathcal E$ and $\mathcal M$ the pseudoeffective and modified nef cones \cite{Bo3} in $H^{1,1}(X,\R)$, respectively. 

\begin{defn}
We say that a class $[\theta]\in\mathcal{E}$ is {\bf plurisupported} if there exists a plurisupported $T\in[\theta]$ 
\end{defn}

If $[\theta]$ is plurisupported, then there will generally be many plurisupported representatives of $[\theta]$, although this is not always the case (e.g. if $[\theta] = c_1(E)$ for an exceptional divisor $E$, in the sense of \cite{Bo3}).

The goal of this subsection will be to show that, on most manifolds, every big $(1,1)$-class admits at least one decomposition into a plurisupported singular class and what we shall call a ``deficient class," which, as the name suggests, will be a ``deficient" in having singular currents (Corollary \ref{decomp}).

We make the following definitions:

\begin{defn}
We write $\mathcal{S}$ for the cone of plurisupported classes. By linearity of the non-pluripolar product, it is a convex cone.
\end{defn}

\begin{rmk}
It is easy to see that $\mathcal S$ contains the cone of effective divisors. As mentioned in the introduction (Question \ref{ookook}) it is unclear if this is all of $\mathcal S$, though we suspect it is not. It is also unclear currently if $\mathcal{S}$  is closed (or open, for that matter). We expect that it is neither, based on the example of the cone of effective divisors.
\end{rmk}

\begin{defn}
We say that a pseudoeffective class $[\alpha]$ is {\bf deficient} if $\langle T\rangle = T$ for all closed positive currents $T\in[\alpha]$. We say that $[\alpha]$ is a {\bf non-trivial} deficient class if $[\alpha]\not= 0$.

We write $\mathcal{A}$ for the cone of all deficient classes.  It is easy to see that $\mathcal{A}\subseteq\mathcal{M}$ (Proposition \ref{prim}).
\end{defn}

As with $\mathcal{S}$, we do not currently know if $\mathcal{A}$ is closed or not (we will see shortly that it is never open).

\begin{exmp}
Many examples of deficient classes can be constructed using dynamical methods to construct rigid classes, which only have a single closed positive current, and then showing that this current has a bounded potential, so that it cannot charge pluripolar subsets \cite{BT2}.

In \cite{Cantat}, Cantat constructs deficient classes on K3 surfaces by examining eigenclasses of automorphisms with positive entroy; he shows that such classes are rigid and that the unique positive current admits continuous potentials. It was later shown that such deficient classes exist in profusion on many K3's, by work of Sibony-Verbitsky \cite{SibVerb} (who show that irrational classes on the boundary of the big cone are rigid) and Filip-Tosatti \cite{FT2} (who show the currents in these classes have continuous potentials).

A similar proof to \cite{Cantat} should also work more generally on hyperkahler manifolds, producing similar examples in higher dimension. The results in \cite{SibVerb} also hold for hyperkahler manifolds, but it does not appear at all trivial to extend the regularity result of \cite{FT2} to hyperkahlers.

There are also other examples coming from foliations of manifolds, see e.g. \cite[Rmk. 3.5]{FT1}.
\end{exmp}

There are two take-aways from the above examples. The first is that, while deficient classes can be viewed as an obstruction to the existence of plurisupported currents (as we will see shortly -- Corollary \ref{decomp}), their existence does not actually preclude the existence of plurisupported currents, as all the examples are projective. Thus, we see right off the bat that deficient classes are an imperfect obstruction at best.

The second take-away is that the classes mentioned are all significantly nicer than one might expect a ``general" deficient class to be; in particular, they are nef and have bounded potential. This prompts the following question:

\begin{quest}\label{failure}
Are deficient classes actually nef? Can one construct a deficient class $[\alpha]$ such that $V_\alpha$ is unbounded?
\end{quest}

The provided examples are all on hyperkahler manifolds or surfaces, which are generally significantly nicer than arbitrary K\"ahler manifolds; in particular, classical Zariski decompositions always exist for such manifolds, which is not true in general. Since the splitting of a class into a plurisupported and deficient class is in some sense ``reverse" to the Zariski decomposition, one should thus probably expect the answer to Question \ref{failure} to be negative in general.

We continue our discussion of deficient classes by showing that they are usually not big.

\begin{prop}\label{prim}
The cone $\mathcal A$ is a union of extremal faces of $\mathcal E$ and $\mathcal M$. In particular, either $\mathcal A = \mathcal E = \mathcal M$ or $\mathcal A\subseteq \partial \mathcal E\cap\partial \mathcal M$.
\end{prop}
\begin{proof}
In order for $\mathcal A$ to be an union of extremal faces of $\mathcal E$, we need to check that if $[\alpha],[\beta]\in\mathcal E$ and $[\alpha + \beta] \in\mathcal A$, then $[\alpha],[\beta]\in\mathcal A$. This is easy to see -- let $T\in[\alpha]$ and $S\in[\beta]$ be closed positive currents. Then $[T + S]\in[\alpha + \beta]$, so by assumption:
\[
\llbracket T + S\rrbracket = \llbracket T\rrbracket + \llbracket S \rrbracket = 0.
\] 
Since $\llbracket T\rrbracket$ and $\llbracket S\rrbracket$ are both positive, they must both be zero.

We now check that $\mathcal A\subseteq \mathcal M$. If $[\alpha]$ is deficient, then:
\[
\llbracket V_\alpha \rrbracket = 0.
\]
Hence, we must have that $\nu(V_\alpha, D)$, the generic Lelong number of $V_\alpha$ along the divisor $D$, must be 0 for every divisor $D\subset X$, by general properties of the (usual) Siu decomposition. By \cite[Prop. 3.6]{Bo3}:
\[
0 \leq \nu(\alpha, D) \leq \nu(V_\alpha, D) = 0\ \ \text{ for all } D,
\]
so that by $[\alpha]\in\mathcal M$ by definition, where $\nu(\alpha, D)$ is the minimal multiplicity of $[\alpha]$ along $D$, as defined by Boucksom \cite{Bo3}.

Finally, since $\mathcal E$ and $\mathcal M$ are both convex cones with non-empty interior, it follows from the definition of an extremal face that either $\mathcal A = \mathcal E = \mathcal M$ or $\mathcal A\subseteq \p\mathcal E\cap\p\mathcal M$.
\end{proof}

\begin{cor}
If $X$ does not admit any plurisupported currents, then every pseudoeffective class on $X$ is primitive. Otherwise, $\mathcal A\subseteq \p\mathcal E\cap\p\mathcal M$. \hfill $\Box$
\end{cor}

From now on, we shall suppose that $[\theta]$ is big. To determine if $[\theta]$ admits a plurisupported current, we investigate the set:
\[
A_\theta := \{ [\langle T\rangle] \in H^{1,1}(X, \R)\ |\ T\in [\theta]\}.
\]
It is easy to see that $A$ will be a convex subset of $(\langle\theta\rangle - \mathcal S)\cap \mathcal{M}$, where we recall that
\[
\langle \theta\rangle = P(\theta) = [\langle T_{\mathrm{min}}\rangle],
\]
for any closed positive current with minimal singularities $T_{\min}\in[\theta]$, since $[\theta]$ is big. Note that in general $A_\theta$ will be neither open nor closed. 

\begin{prop}\label{Frasier}
We have:
\[
A_\theta \supseteq \left((\langle\theta\rangle - \mathcal S)\cap \mathcal M\cap\mathcal E^\circ\right) \cup \left((\langle \theta\rangle - \mathcal S)\cap \mathcal A\right),
\]
for any big class $[\theta]$ on $X$.
\end{prop}
\begin{proof}
If $X$ admits no plurisupported currents, then this is obvious. So suppose that $T$ is a plurisupported current on $X$ and let $[\gamma] := \langle\theta\rangle - [T]$. We will show that $[\gamma]$ admits a current $S\in[\gamma]$ with $S = \langle S \rangle$ if either $[\gamma]\in \mathcal{A}$ or if $[\gamma]\in\mathcal{M}\cap\mathcal{E}^\circ$; assuming this, we will be done, as $S + T$ will be a closed positive current in $\langle\theta\rangle$ with $[\langle S + T\rangle] = [S] = [\gamma]$.

The first case is obvious, so suppose that $[\gamma] \in \mathcal M\cap \mathcal E^\circ$. As mentioned right before Proposition \ref{I-sings}, $V_\gamma$ will always be $\mathcal{I}$-model, and hence $\llbracket V_\gamma\rrbracket$ will be a sum of divisors; let $S := \gamma_{V_\gamma}$. As a standard consequence of the Siu decomposition, we have that:
\[
\llbracket S\rrbracket = \sum \nu(V_\gamma, D_i) \llbracket D_i\rrbracket;
\]
since $[\gamma]$ is big and modified nef, by \cite[Prop. 3.6]{Bo3} $\nu(V_\gamma, D_i) = 0$ though, so we have $\llbracket S\rrbracket = 0$, or equivalently $\langle S \rangle = S$.
\end{proof}

Recall that the relative interior of a convex set $B$ is defined to be:
\[
B^\circ = \{\lambda x + (1-\lambda)y\ |\ \lambda\in(0,1), x,y\in B\}.
\]

\begin{cor}
We have:
\[
A_\theta^\circ = (\langle\theta\rangle - \mathcal S^\circ)\cap \mathcal M^\circ.
\]
\end{cor}
\begin{proof}
This is immediate from the definition and Proposition \ref{Frasier}.
\end{proof}

Consider now $\overline{A_\theta}$ for a big class $[\theta]$ -- this is a compact set, and so has a minimal element under the partial ordering induced by $\mathcal E$ (n.b. it may have many, incomparable minima).

\begin{prop}
Any minima of $\overline{A_\theta}$ is deficient. Consequently, any minima of $\overline{A_\theta}$ are also minima of $A_\theta$.
\end{prop}
\begin{proof}
Let $[\alpha]$ be a minima of $\overline{A_\theta}$ and suppose that $T\in[\alpha]$.

By the definition of $\overline{A_\theta}$, there exists a sequence of closed positive currents $S_i\in[\theta]$ such that:
\[
\lim_{i\rightarrow\infty} [\langle S_i\rangle] = [\alpha].
\]
Pick a sequence of positive numbers $c_i\rightarrow 0$ such that $[\alpha - \langle S_i\rangle]\leq c_i\langle\theta\rangle$, which we can do since $\langle \theta\rangle$ is big. Then there exist positive currents:
\[
S'_i\in c_i\langle\theta\rangle - [\alpha - \langle S_i \rangle],
\]
and we have:
\[
T + S'_i \in [\langle S_i\rangle] + c_i \langle\theta\rangle.
\]
Define $R_i\in(1+c_i)[\theta]$ by $R_i := T + S_i' + \llbracket S_i\rrbracket$, so that $\langle R_i\rangle = \langle T\rangle + \langle S_i '\rangle$. Since
\[
0\leq \lim_{i\rightarrow\infty}\langle S'_i\rangle \leq \lim_{i\rightarrow \infty} S'_i = 0,
\]
we have that:
\[
[\langle T\rangle] = \lim_{i\rightarrow\infty} \frac{1}{1+c_i}[\langle R_i\rangle] \in \overline{A_\theta}
\]
by definition. By definition, $[\alpha] - [\langle T \rangle] = [\llbracket T\rrbracket]$. Since $[\alpha]$ is a minima, we must have $\llbracket T\rrbracket = 0$, as desired.
\end{proof}

\begin{cor}\label{decomp}
\[
\mathcal E^\circ \subseteq \mathcal A + \mathcal S.
\]
Alternatively, any big class admits at least one decomposition into the sum of a deficient class and a plurisupported class.
\end{cor}

\begin{cor}
If $X$ has no non-trivial deficient classes, then every big class admits a plurisupported current.
\end{cor}

As suggested before, it is unclear if such manifolds even exist, and even if they did, verifying this condition seems particularly challenging without further tools.


\subsection{Currents with Maximal Singularity Type}


We discuss a starting point to a possible alternate approach to determine the existence of plurisupported currents. Inspired by Corollary \ref{ordering}, we make the following definition:

\begin{defn}
Suppose that $[\theta]$ is a pseudoeffective $(1,1)$-class. We say that a function $\vp \in\PSH(X,\theta)$ has {\bf maximal singularity type} if there does not exist any $\psi\in\PSH(X,\theta)$ such that $[\psi]$ is strictly more singular than $[\vp]$. We employ similar terminology for the corresponding current $\theta_\vp$.
\end{defn}

A simple application of Lemma \ref{subtraction} shows that if $T\in[\theta]$ is plurisupported, then $T$ has maximal singularity type in $[\theta]$; hence, one possible idea to determine if/when plurisupported currents exist is to try and characterize currents with maximal singularity type. Unfortunately, it is not clear that currents with maximal singularity type exist either -- hence, we make another definition:

\begin{defn}
Suppose that $[\theta]$ is a pseudoeffective $(1,1)$-class. We say that a function $\vp \in\PSH(X,\theta)$ has {\bf maximal model singularity type} if it has model singularity type, in the sense of \cite{DDL4} and if further there does not exist any $\psi\in\PSH(X,\theta)$ with model singularity type, such that $[\psi]$ is strictly more singular than $[\vp]$. We employ similar terminology for the corresponding current $\theta_\vp$.
\end{defn}

It is easy to see that plurisupported $\vp$ also have model singularity type (Proposition \ref{ignore}), and are hence also maximally model (in general though, it is unclear if the notions of maximal singularity type and maximal model singularity type should agree).

The reason for making this alternate definition is that one can show that maximal model singularity types always exist in big classes, thanks to \cite{DDL4}:

\begin{prop}
Suppose that $[\theta]$ is a big cohomology class. Then there exists some $\vp\in\PSH(X,\theta)$ with maximal model singularity type.
\end{prop}
\begin{proof}
By \cite[Thrm. 3.3]{DDL4}, the classes in $(\mathcal{S}, d_\mathcal{S})$ are in one-to-one correspondence with the $\vp\in\PSH(X,\theta)$ with model singularity type. It is also clear that $[\vp] \leq [\psi]$ if and only if we have the pointwise inequality $\mathcal{C}_\theta(\vp) \geq \mathcal{C}_\theta(\psi)$. Moreover, we also clearly have $\sup_X \mathcal{C}_\theta(\vp) = 0$ for any $[\vp]$. 

Thus, if $[\vp_i]$ is a sequence of singularity types such that $[\vp_i] \leq [\vp_{i+1}]$, we see that the sequence $\mathcal{C}_\theta(\vp_i)$ is a decreasing sequence of sup-normalized $\theta$-psh functions, and hence has a limit $\vp\in\PSH(X,\theta)$. By \cite[Cor. 4.7]{DDL4}, $\vp = \mathcal{C}_\theta(\vp)$, and clearly $[\vp_i] \leq [\vp]$ for any $i$. Hence, by Zorn's Lemma, we conclude there exists at least one maximally model singularity type in $(\mathcal{S}, d_\mathcal{S})$, as claimed.

\end{proof}

Unfortunately, characterizing maximally model singularity types appears to be quite difficult. For instance, one naive idea is to ask if they can be characterized by the vanishing of their non-pluripolar product; an easy example on $\CP^1\times\CP^1$ shows this is false.

\begin{exmp}
Let $x\in \CP^1$ and $\psi_x$ be the unique $\omega_{FS}$-metric with $\omega_{FS} + i\p\pbar\psi_x = \delta_x$. Let $\pi_1, \pi_2$ be the natural projections from $\CP^1\times\CP^1$. Then $\omega := \pi_1^*\omega_{FS} + \pi^*_2\omega_{FS}$ is a K\"ahler form, and we see:
\[
\int_{\CP^1\times\CP^1} \langle (\omega + i\p\pbar\pi_1^*\psi_x)^2\rangle = \int_{\CP^1\times\CP^1} \pi_2^*\omega_{FS}^2 = 0.
\]
But $\pi_1^*\psi_x$ is clearly not maximally singular, as, for instance, $\pi_1^*\psi_x + \pi^*_2\psi_x\in\PSH(\CP^1\times\CP^1, \omega)$ is more singular than it.
\end{exmp}

Thus, there can be whole hierarchies of singularity types with no Monge-Amp{\`e}re mass. Already studying such singularity types has proven to be quite difficult, since many of the standard tools and techniques in pluripotential theory fail to provide useful information without the positive mass assumption.

\end{document}